\documentclass[11pt,a4paper]{article}

\usepackage[normalem]{ulem} 
\usepackage{pifont}
\usepackage{phaistos}
\usepackage{marvosym}
\usepackage{latexsym}
\usepackage{psfrag}
\usepackage{amsmath,amssymb,cite}
\usepackage{latexsym}
\usepackage{graphicx}
\usepackage{lscape}
\usepackage{color}
\usepackage{xcolor}
\usepackage{soul}
\usepackage{afterpage}

\usepackage{hyperref}
\hypersetup{
    colorlinks=true,
    linkcolor=blue,
    filecolor=magenta,      
    urlcolor=cyan,
    citecolor=red,
}
\DeclareMathOperator*{\sgn}{sgn}

\newcommand{\mut}{\widetilde{\mu}}
\newcommand{\at}{\widetilde{a}}
\newcommand{\bt}{\widetilde{b}}
\newcommand{\ds}{\displaystyle}
\newcommand{\nexto}{\kern -0.54em}
\newcommand{\dR}{\mathbb{R}}
\newcommand{\dN}{\mathbb{N}}

\newcommand{\proofbox}{\hspace{\fill}{$\Box$}}

\newtheorem{lemma}{Lemma}
\newtheorem{theorem}{Theorem}
\newtheorem{corollary}{Corollary}
\newtheorem{definition}{Definition}
\newtheorem{proposition}{Proposition}

\newtheorem{remark}{Remark}

\newtheorem{algorithm}{Algorithm}
\newenvironment{proof}{Proof.}{\proofbox}

\begin{document}

\author{Authors}

\author{Orhan Ar{\i}kan\footnote{Electrical and Electronics
    Engineering Department, Bilkent University, Bilkent, 06800 Ankara,
    Turkey.  E-mail: oarikan@ee.bilkent.edu.tr\,.}
\and
Regina S. Burachik\footnote{School of Information Technology and
  Mathematical Sciences, University of South Australia, Mawson Lakes,
  S.A. 5095, Australia. 
  E-mail: regina.burachik@unisa.edu.au\,.}
\and
C. Yal{\c c}{\i}n Kaya\footnote{School of Information Technology and
  Mathematical Sciences, University of South Australia, Mawson Lakes,
  S.A. 5095, Australia.
  E-mail: yalcin.kaya@unisa.edu.au\,.}}

\title{\vspace{-10mm}\bf Steklov Regularization and Trajectory Methods for Univariate Global Optimization}
\date{\today}
\maketitle

\thispagestyle{empty}

\begin{abstract} 
{\sf We introduce a new regularization technique, using what we refer to as the Steklov regularization function, and apply this technique to devise an algorithm that computes a global minimizer of univariate coercive functions.  First, we show that the Steklov regularization convexifies a given univariate coercive function.  Then, by using the regularization parameter as the independent variable, a trajectory is constructed on the surface generated by the Steklov function.  For monic quartic polynomials, we prove that this trajectory does generate a global minimizer.  In the process, we derive some properties of quartic polynomials.  Comparisons are made with a previous approach which uses a quadratic regularization function. We carry out numerical experiments to illustrate the working of the new method on polynomials of various degree as well as a non-polynomial function.}
\end{abstract}

\begin{verse} {\em Key words}\/: {\sf Global optimization, mean filter, Steklov smoothing, Steklov regularization, scale-shift invariance, trajectory methods.}
\end{verse}

\pagestyle{myheadings}
\markboth{}{\sf\scriptsize Steklov Regularization for Univariate Global Optimization \ \ by O. Ar{\i}kan, R. S. Burachik and C. Y. Kaya}

\section{Introduction}

Mean filter is a digital filtering technique in signal processing, which is used to remove noise.  The technique can also be viewed as a smoothing procedure.  In digital imaging, for example, this filtering technique is performed by replacing each pixel value by the mean value of its neighbours and itself in a ``window'' -- see~\cite{ZhaXio2009} and the references therein.  The expected outcome is the removal of noise in the image and the smoothening of the image.  The mean filter idea was originally proposed and has so far been used for the processing of discrete data.

In the present paper, we propose and analyse a similar idea in the setting of continuous optimization, involving a coercive univariate function instead of discrete data.  When the averaging process described above is employed for a univariate function $f(x)$ over an interval (corresponding to a window) of variable size centred at $x$, i.e., $[x-t,x+t]$, with $t>0$, one obtains the well-known Steklov smoothing function~\cite{Chen2012, GarVic2013, Gupal1977, ErmNorWet1995}, expressed in terms of the function $f$ and the size of the interval, denoted here by $\mu(x,t)$.  The Steklov smoothing function is typically used in getting an approximate solution to the problem of minimizing a nonsmooth $f$:  The smooth function $\mu(x,t)$ is minimized over an interval, or window, with small $t$, so that the solution of the smoothed problem is a close enough approximation to the solution of the original problem.  

Although the properties of $\mu$ have very well been explored in the literature for small $t$, it has not yet been studied for large $t$. This is the point where our study steps in.  In the present paper, first we show that for large enough $t$ and certain coercive $f$, $\mu(\cdot,t)$ is strictly convex -- see Theorem~\ref{convexity}.  In this sense,  $\mu$ {\em regularizes} the function $f$ for large $t$ by convexifying (as well as smoothening) it, and that is the reason why we call it the {\em Steklov regularization function}.  We note that, if $\mu(\cdot,t_0)$ is strictly convex for some $t_0>0$, then $\mu(\cdot,t_0)$ has a unique minimizer.  The main aim of the current paper is to propose a method, namely Algorithm~\ref{algo1}, based on constructing and following a trajectory between the unique minimizer of $\mu(\cdot,t_0)$ and a global minimizer of $f(x)$.  The trajectory here is the solution of an ordinary differential equation (ODE) obtained by using $\mu(x,t)$.

Univariate global optimization has long been an active area of research -- see \cite{HorTuy1996, LerSer2013, Scholz2012} and the references therein.  %
Most multidimensional iterative methods involve line searches at a given search direction, and these uni-dimensional searches are equivalent to the global minimization of a univariate function.  Therefore, finding efficiently a global minimizer in such a line search has importance on its own, and can be crucial for the success of such iterative high-dimensional techniques. Thus the relevance of developing new, efficient univariate techniques. Moreover, the result of such a line search can be useful as a starting guess for the global minimizer of the original higher dimensional problem. Although the present paper focuses on the univariate case, an extension of our approach to the multi-variable case is under investigation.

Trajectory based methods are not new to optimization.  The trajectories (to follow) in these methods are typically solutions of ODEs incorporating the gradient of $f(x)$.  Convergence analyses for these types of methods have so far been given only for local minima -- see, for example, \cite{AttChbPeypRed2018, BotCse2018} and the references therein.  Trajectory based methods have been proposed also for global optimization, albeit without a convergence proof, to the best knowledge of the authors -- see, for example, \cite{SnyKok2009}.

We note one particular trajectory based technique for global optimization, the {\em backward differential flow} method, which was proposed by Zhu et al.\ in~\cite{ZhuZhaLiu2014}, where the trajectories are solutions of an ODE that emanates from the (classical) quadratic regularization function rather than the Steklov regularization function.  We have recently illustrated that the backward differential flow method, given as Algorithm~\ref{algo3} in the current paper, may not yield a global minimizer, even in the case when the function is a quartic polynomial -- see~\cite{AriBurKay2015}.  

We provide a convergence proof of our approach for the case of quartic polynomials (see Algorithm~\ref{algo2} and Theorem~\ref{well-defined}). Our numerical experiments indicate that our method can be viewed as a better alternative to that given by Zhu et al.~\cite{ZhuZhaLiu2014}.  Indeed, our method converges to a global minimum in most of the (randomly generated) cases of even-higher-degree monic polynomials. On the other hand, the method by Zhu et al.\ fails to converge in the great majority of the cases -- see Table~\ref{failure_rates}.

In addition to the convexification and convergence results in Theorems~\ref{convexity} and \ref{well-defined}, respectively, we provide auxiliary results, which are interesting in their own right.  For example, we prove in Lemma~\ref{shift_lemma} that, if Algorithm~\ref{algo1} can generate the global minimizer of a given function $f(x)$, then it can also generate the global minimizer of $f(\alpha\,x - a)$, where $a\in \dR$ and $\alpha>0$ are fixed.  We refer to this property as the {\em scale-shift invariance} property.  We note that, while Algorithm~\ref{algo1} (and thus Algorithm~\ref{algo2}) is scale-shift invariant, Algorithm~\ref{algo3} is not.  It is well-known that scale changes and translations can be used to simplify the expression of a function. For example, the third degree term of a quartic polynomial can be made to vanish after a simple horizontal shift, which transforms the polynomial into the so-called {\em depressed} form.

We uncover certain properties of quartic polynomials, which are independent of the method we propose.  Lemma~\ref{curvature} states that, if a quartic polynomial has two local minimizers, then its curvature at the global minimizer is greater.  Moreover, its global minimizer is farther from the origin.  Lemma~\ref{global_sign}, on the other hand, tells us at least how far from the origin the global minimizer will be located and what its sign is going to be.  Lemma~\ref{f: quasi-convex} presents a simple condition under which a monic depressed quartic polynomial is quasi-convex.  Lemma~\ref{t0_x0} asserts the value $t_0$ that convexifies $\mu(\cdot,t_0)$ and states the minimizer $x_0$ of $\mu(\cdot,t_0)$, which are conveniently used in Algorithm~\ref{algo2}.  Proposition~\ref{mu:quasi-convex} provides a condition on $t$ for quasi-convexity of $\mu(\cdot,t_0)$.  Lemmas~\ref{lem:solvability}--\ref{lem:3} provide some properties of the trajectories run in Algorithm~\ref{algo2} which in turn facilitate the proof of Theorem~\ref{well-defined}.

The paper is organized as follows.  In Section~2, we introduce the Steklov regularization and prove certain properties, including convexification.  In Section~3, we describe Algorithm~\ref{algo1} and prove the scale-shift invariance property.  In Section~4, we derive some properties of quartic polynomials, provide Algorithm~\ref{algo2} and prove its convergence.  In Section~5, we describe Algorithm~\ref{algo3}, which uses the quadratic regularization.  In Section~6, we carry out extensive numerical experiments using Algorithms~\ref{algo1} and \ref{algo3} for polynomials of various degrees, including a non-polynomial example, and make comparisons.


\section{Steklov Regularization }

In an analogous way to the original (discrete) mean filter technique~\cite{ZhaXio2009}, first choose a ``window'' with centre $x$.  In the case when $x$ is a scalar, this window is just a finite interval.  Then compute the {\em mean value} of a continuous function $f:\dR\to\dR$ over the window and assign this value as the value of an associated function at $x$.  Furthermore, pass/shift the window across the whole domain of $f$, assigning values to the mean function at every $x$ in the domain of~$f$.

The window, or the interval, can typically be chosen to be centred at $x$, as $[x-t,x+t]$, where $t$ is a fixed positive real number defining the window size.  Therefore, we can regard the associated function as a function of not only $x$ but also $t$.  

The function we have just motivated with mean filter turns out to be already in use in the nonsmooth optimization literature, in obtaining smooth approximations of nondifferentiable objective functions, via a convolution integral, for $t$ small enough.  A well-known class of {\em mollifiers} in the convolution integral is referred to as the Steklov mollifiers \cite{GarVic2013}.  A use of these mollifiers in the convolution integral in turn gives rise to the so-called {\em Steklov smoothing function}, definition and properties of which can be found in \cite{Chen2012, GarVic2013, Gupal1977, ErmNorWet1995}.  

We note that the function we have motivated by means of mean filter is nothing but the Steklov smoothing function.  Since our concern will be to {\em convexify} a given function for large enough $t$ (rather than making it smooth for small $t$), we refer to the resulting function as the {\em Steklov regularization function}, or simply the {\em Steklov function}.

\begin{definition} \rm
The {\em Steklov function} associated with a continuous function $f$ is denoted by $\mu:\dR\times (0,\infty)\to \dR$ and defined as
\begin{equation}  \label{mf-func}
\mu(x,t) := \frac{1}{2t}\,\int_{x-t}^{x+t} f(\tau)\,d\tau\,.
\end{equation}
We also refer to $\mu(\cdot,\cdot)$ as the {\em Steklov regularization of} $f$.
\endproof
\end{definition}

\begin{remark} \rm
Since the function $f$ is continuous, $\mu:\dR\times [0,\infty)\to\dR$ is well defined and  differentiable on $\dR\times (0,\infty)$.
\endproof
\end{remark}

We collect in the next lemma some useful properties of $\mu$. 


\begin{lemma}\label{L-prop-mu}
Given a continuous function $f:\dR\to\dR$, let $\mu:\dR\times (0,\infty)\to \dR$ be as in \eqref{mf-func}. The following equalities hold for $\mu$. 
\begin{itemize}
\item[(i)]\begin{equation}  \label{mux-1}
\mu_x(x,t) = \frac{1}{2t}\,(f(x+t) - f(x-t)),
\end{equation}
where $\mu_x$ stands for $\partial\mu/\partial x$.
\item[(ii)] \begin{equation}  \label{muxx}
\mu_{xx}(x,t) = \frac{1}{2t}\,(f'(x+t) - f'(x-t))\,,
\end{equation}
where $\mu_{xx}$ stands for $\partial^2\mu/\partial x^2$.
\item[(iii)]  \begin{equation} \label{mutx}
\mu_{tx}(x,t) = \frac{1}{t}\,\left[\frac{1}{2}\,(f'(x+t) + f'(x-t)) - \mu_x(x,t)\right]\,, 
\end{equation} 
where $\mu_{tx}$ stands for $\partial^2\mu/\partial t\,\partial x$.
\end{itemize}
\end{lemma}
\begin{proof}
Part (i) follows directly from the Fundamental Theorem of Calculus, and the remaining parts are obtained by differentiating $\mu_x$ with respect to $x$ and $t$, respectively.
\end{proof}

The following theorem states general assumptions under which the Steklov function $\mu$
convexifies a coercive function $f$, and hence we regard the effect of $\mu$ as a regularization.

\begin{theorem} [Convexification] \label{convexity} Suppose that $f:\dR\to \dR$ is a continuously differentiable function such that there exist two real numbers $a$ and $b$, with $a < b$, for which the following conditions hold.
\begin{itemize}
\item[(a)] $f'(x)<0$ for all $x\le a$ and $f'(x) >0$  for all $x\ge b$.
\item[(b)] $f'$ is strictly increasing and unbounded below on $(-\infty,a]$.
\item[(c)] $f'$ is strictly increasing and unbounded above on $[b,\infty)$.
\end{itemize}
Then there exists $t_0 > 0$ such that $\mu(\cdot,t)$  is strictly convex for all $t \ge t_0$.
\end{theorem}
\begin{proof}
From part (a) and the fact that $f'$ is continuous on $[a,b]$, there exist real numbers $\alpha < 0$ and $\beta > 0$ such that
\[
\alpha:=\min_{x\in [a,b]} f'(x)\le f'(a)<0<f'(b)\le \max_{x\in [a,b]} f'(x) =: \beta\,.
\] 
By parts~(b) and (c), there exist $\at\le a$ and  $\bt\ge b$ such that $f'(x) < \alpha$  for all $x\le\at$, and  $f'(x) > \beta$ for all $x\ge\bt$.  Let $t_0 \ge \bt - \at>0$.  We will show that for every $t\ge t_0$ and every $x\in \dR$, we have
\begin{equation}  \label{CC}
f'(x+t) - f'(x-t) > 0.
\end{equation}
By \eqref{muxx}, this amounts to showing convexity of $\mu(\cdot,t)$ all $t \ge t_0$. Only the following cases are possible for the pair $x-t_0,\,x+t_0$.  
\begin{itemize}
\item[(i)] $x - t_0, x + t_0 \in (-\infty,\at]$.
\item[(ii)] $x - t_0, x + t_0 \in [\bt,\infty)$.
\item[(iii)] $x - t_0 \in (-\infty,\at]$ and $x + t_0 \in (\at,\bt)$.
\item[(iv)] $x - t_0 \in (\at,\bt)$ and $x + t_0 \in [\bt,\infty)$.
\item[(v)] $x - t_0 \in (-\infty,\at]$ and $x + t_0 \in [\bt,\infty)$.
\end{itemize}
Note that the case $x - t_0,\,x+t_0 \in (\at,\bt)$ is not possible by the choice of $t_0$. Indeed, if $x - t_0,\,x+t_0 \in (\at,\bt)$ we can write
\[
t_0<x-\at \hbox{\ \ \ and\ \ \ } t_0<\bt-x\,,
\]
so $t_0 < (\bt-\at)/2 < \bt-\at$, contradicting the choice of $t_0$. We prove \eqref{CC} by considering all the possible cases (i)--(v).  
\begin{itemize}

\item[(i)] By part (b) and the fact that $x+t_0 > x-t_0$, we have that $f'(x+t_0) - f'(x-t_0) > 0$. To complete the proof of \eqref{CC}, fix now $t>t_0$. We have the following sub-cases:
\begin{center}
(i$_1$)~$x+t \in (\at,a)$\,,\qquad (i$_2$)~$x+t \in [a,b]$\,,\qquad (i$_3$)~$x+t \in (b,+\infty)$\,.  
\end{center}
In case (i$_1$) we use part (b) and the fact that  $x-t_0,\,x-t, x+t_0,\,x+t\in (-\infty,a]$ to write
\begin{eqnarray*}
 f'(x+t_0) & < & f'(x+t)\,,    \\
  f'(x-t_0)  &  > & f'(x-t)\,,
\end{eqnarray*}
so $0< f'(x+t_0) - f'(x-t_0)< f'(x+t) - f'(x-t)$, as desired. In case  (i$_2$), we use the definition of $\alpha$ to write $f'(x+t)\ge \alpha$. Since $x + t_0 \in (-\infty,\at]$ we also have that $f'(x+t_0)< \alpha$. Using part~(b) and the fact that $x-t<x-t_0\le \at$, we have
\[
0< f'(x+t_0) - f'(x-t_0) < \alpha- f'(x-t) \le f'(x+t)- f'(x-t)\,,
\]
as desired. In sub-case (i$_3$), we note that $f'(x+t)>\alpha$. Indeed, since $x+t>b$ we use part~(c) to write $\alpha\le f'(b)<f'(x+t)$. Altogether,
\[
0< f'(x+t_0) - f'(x-t_0) < \alpha- f'(x-t_0) < f'(x+t)- f'(x-t)\,,
\]
where we also used (b) in the third inequality. This completes the proof for case (i). Due to symmetry, the proof for case (ii) is done in exactly the same way as for case (i), {\em mutatis mutandis}. We therefore omit the proof for case (ii).

\item[(iii)]  As in (i), we consider three subcases: 
\begin{center}
(iii$_1$)~$x+t_0\in (\at,a)$\,,\qquad (iii$_2$) $x+t_0 \in [a,b]$\,,\qquad (iii$_3$) $x+t_0 \in (b,+\infty)$\,.
\end{center}
Case (iii$_1$) implies that $x-t_0, x+t_0\in (-\infty,a]$ and by part~(b) 
\begin{equation}\label{C1}
f'(x+t_0)-f'(x-t_0) > 0\,.
\end{equation}
Case (iii$_2$) gives $x+t_0\in [a,b]$ and $x-t_0\in (-\infty,\at]$. So, again we have \eqref{C1}. Indeed,
\[
f'(x+t_0)-f'(x-t_0) > \alpha -\alpha=0\,.
\]
In case (iii$_3$) we have  $x+t_0\in (b,\infty)$ and $x-t_0\in (-\infty,\at]$. So by parts (b) and (c) we have that $\alpha\le f'(b)<f'(x+t_0)$ and  $f'(x-t_0) <\alpha$. As in case  (iii$_2$) we obtain \eqref{C1}. \\
To complete the proof for case~(iii), fix $t\ge t_0$. As in case (i) we need to consider three sub-cases: 
\begin{center}
(iii$_4$)~$x+t\in (\at,a)$\,,\qquad (iii$_5$)~$x+t \in [a,b]$\,,\qquad (iii$_6$)~$x+t \in (b,+\infty)$\,.
\end{center}
All three sub-cases are resolved exactly as in cases (iii$_1$), (iii$_2$) and (iii$_3$), respectively, with $t_0$ replaced by $t$. This completes the proof for case (iii).
 
\item[(iv)]  Again, we consider three sub-cases: 
\begin{center}
(iv$_1$)~$x-t_0\in (\at,a)$\,,\qquad (iv$_2$)~$x-t_0 \in [a,b]$\,,\qquad (iv$_3$)~$x-t_0 \in (b,\bt)$\,.
\end{center}
In case (iv$_1$) we have $\at<x-t_0<a$ so by part~(b) $f'(x-t_0)<f'(a)\le \beta$.  Also in case~(iv$_2$) we have $f'(x-t_0)\le \beta$. In both cases, we can write
\[
f'(x+t_0)-f'(x-t_0) >\beta -\beta=0\,,
\]
where we also used the fact that $f'(x+t_0)>\beta$. In case (iv$_3$), $x-t_0,\, x+t_0\in (b,\infty)$ and we use directly part~(c) to conclude that  $f'(x+t_0)-f'(x-t_0) >0$.  \\
To complete the proof for case~(iv), fix $t\ge t_0$. We always have that $x+t\in [\bt,\infty)$ so $f'(x+t)>\beta$. We consider again the following sub-cases: 
\begin{center}
(iv$_4$)~$x-t\in (\at,a)$\,,\qquad (iv$_5$)~$x-t \in [a,b]$\,,\qquad (iv$_6$)~$x-t_0 \in (b,\bt)$\,.
\end{center}
As in case (iii), all three sub-cases are resolved exactly as cases (iv$_1$), (iv$_2$) and (iv$_3$),  respectively, with $t_0$ replaced by $t$. This completes the proof for case (iv).

\item[(v)]  Use parts~(b) and (c) to write
  \[
  0<\beta -\alpha < f'(x+t_0) - f'(x-t_0)<f'(x+t) - f'(x-t),
  \]
  where we used the definition of $\at$ and $\bt$ in the second inequality and parts~(b) and (c) in the third.  This completes the proof for case~(v).
\end{itemize}
The proof of the theorem is complete.
\end{proof}

\begin{remark}\label{R1} \rm
It is easy to check that, in Theorem~\ref{convexity}, we can take $a:=-R$ and $b:=R$ for $R:=\max\{|a|, |b|\}$. Note that $R>0$ because $a<b$.
\end{remark}

\begin{remark} \rm
Monic polynomials of even degree are an important special case of functions which can be convexified by $\mu$.
\end{remark}

Graphical depictions of $\mu(x,t)$ for typical monic quartic polynomials and how convexification happens in each example case can be observed in Figure~\ref{curve_types} on page~\pageref{curve_types} .

We focus our attention on functions which are {\em coercive}, in the sense of~\cite[Definition~3.25]{RocWet2004}. In our framework, this concept is stated as follows.

\begin{definition}  \label{def:coercive} \rm
 Let  $f:\dR\to\dR$ be a function which is bounded below on bounded sets. We say that $f$ is {\em
   coercive} if
\begin{equation}  \label{eq:10}
   \liminf_{|x|\to \infty}\displaystyle \frac{f(x)}{|x|}=\infty\,.
\end{equation}
\endproof
\end{definition}

Coercive functions might be non-differentiable, and hence in general they may not verify the assumptions of Theorem~\ref{convexity}. The following result shows that a function verifying the assumptions of Theorem~\ref{convexity} is coercive in the sense of Definition~\ref{def:coercive}.

\begin{proposition}[Coercivity]  \label{coercive}
Let $f$ be as in Theorem~\ref{convexity}. Then $f$ is coercive in the sense of Definition~\ref{def:coercive}.
\end{proposition}
\begin{proof}
The statement on the boundedness of $f$ is a direct consequence of the continuity of $f$. By Remark~\ref{R1}, we can assume that Theorem~\ref{convexity} holds with $a=-R$ and $b=R$ for some $R>0$.  To prove \eqref{eq:10} we will show that for all $M>0$ 
 we have 
 \begin{equation}\label{eq:12}
 \liminf_{|x|\to \infty}\displaystyle \frac{f(x)}{|x|}\ge M.
\end{equation}
If \eqref{eq:12} is not true, there exists $M_0>0$ and a sequence $(x_n)\subset \dR$ such that $|x_n|>n$ for all $n\in \dN$ and 
\begin{equation}\label{eq:13}
 \frac{f(x_n)}{|x_n|}<M_0.
\end{equation}
Without loss of generality we can assume that the sequence $(x_n)\subset [R,\infty)$ and strictly monotone increasing, or $(x_n)\subset (-\infty,-R]$ and strictly monotone decreasing (otherwise we take a subsequence of the original sequence). Moreover, we can further assume that $(x_n)\subset [R,\infty)$ and strictly monotone increasing, because the proof for the latter case is identical to the one for the case in which $(x_n)\subset (-\infty,-R]$ and strictly monotone decreasing ({\em mutatis mutandis}).  So it is enough to assume that $(x_n)\subset [R,\infty)$ and strictly monotone increasing. Since $x_n\uparrow +\infty$ and $f'$ is strictly increasing and unbounded above in $[R,\infty)$ there exists $n_0$ such that $f'(x_n)>2M_0$ for all $n\ge n_0$. Using the mean value theorem we can write for all $n> n_0$:
\[
 \begin{array}[h]{rcl}
f(x_{n})-f(x_{n_0})&=&  \ds\sum_{j=n_0}^{n-1} \left(f(x_{j+1})-f(x_j)\right)= \sum_{j=n_0}^{n-1}  f'(\theta_j)(x_{j+1}-x_j)\\
&&\\
 &>& f'(x_{n_0}) \ds\sum_{j=n_0}^{n-1} (x_{j+1}-x_j)> 2M_0 (x_n-x_{n_0}), \\
 \end{array}
\]
where we used that $R\le x_{n_0}\le x_j<\theta_j< x_{j+1}$ and the fact that $f'$ is increasing in $[R,\infty)$ in the first inequality, 
and the definition of ${n_0}$ in the last one. Dividing the expression by $|x_n|=x_n$ and using \eqref{eq:13} we obtain
\[
\begin{array}[h]{rcl}
M_0 -  \ds\frac{ f(x_{n_0})}{x_n}&>& \ds\frac{ f(x_{n})}{x_n}-  \frac{ f(x_{n_0})}{x_n}> 2M_0 \left(1- \frac{x_{n_0}}{x_n}\right).
 \end{array}
\]
Taking limits for $n\to \infty$ and using the fact that $x_n\to +\infty$  we obtain
\[
\begin{array}[h]{rcl}
M_0&\ge& 2M_0,
 \end{array}
\]
a contradiction.  This completes the proof.
\end{proof}

The next proposition shows that $\mu(\cdot,t)$ is a good approximation of $f$ at $x$ for small values of $t$.

\begin{proposition}[Limiting Functions]  \label{limits}
Fix $x\in \dR$ and  $t_0>0$. Assume that $f$ is twice continuously differentiable at $x$.
\begin{equation}  \label{Prop1}
\lim_{t\to 0}\mu(x,t) = f(x)\,,\ \ \lim_{t\to 0}\mu_x(x,t) = f'(x)\,,\ \ \lim_{t\to 0}\,\mu_{xx}(x,t) = f''(x)\,,\ \ \lim_{t\to 0}\,\mu_{tx}(x,t) = 0\,.
\end{equation}
\end{proposition}
\begin{proof}
The first limit is a consequence of l'H{\^ o}pital's rule:
\[
\lim_{t\to 0}\mu(x,t) = \lim_{t\to 0}\frac{\ds\int_{x-t}^{x+t} f(\tau)\,d\tau}{2\,t}= \lim_{t\to 0} \frac{f(x+t) +
f(x-t)}{2} = f(x)\,.
\]
The second and third limits are also a result of the application of l'H{\^ o}pital's rule on the limit, as $t\to 0$, of \eqref{mux-1} and \eqref{muxx}, respectively.  Proving the last equality is more involved: Since $f$ is ${\cal C}^2$ we can write
\begin{eqnarray*}
&& f(x+t) = f(x)+ tf'(x)+o_1(t^2)\,, \\[1mm]
&& f(x-t) = f(x)- tf'(x)+o_2(t^2)\,,
\end{eqnarray*}
where $\lim_{t\to 0} (o_i(t^2)/t) = 0$, for $i=1,2$. Using the two equalities above we derive
\begin{eqnarray}
\mu_{x}(x,t) &=& \ds\frac{1}{2t}\,(f(x+t) - f(x-t))= \frac{1}{2t}\,(2tf'(x) + o_3(t^2)) \nonumber \\
&=& f'(x) + o(t)\,,  \label{m1}
\end{eqnarray}
where $\lim_{t\to 0} (o_i(t)/t) = 0$.  Using \eqref{m1} we have
\begin{eqnarray*}
\lim_{t\to 0} \mu_{tx}(x,t) &=&  \lim_{t\to 0} \frac{1}{2t}  \left[f'(x+t)+f'(x-t)\right] -  \frac{1}{t}\mu_x(x,t) \\[1mm]
&=& \lim_{t\to 0}\frac{1}{2t}\left[f'(x+t)+f'(x-t)\right] -  \frac{1}{2t}f'(x)  -  \frac{1}{2t}f'(x)- \frac{o(t)}{t} \\[2mm]
&=& \lim_{t\to 0}\frac{1}{t}  \left[ \frac{f'(x+t)-f'(x)}{t}- \frac{[f'(x-t) -f'(x)]}{(-t)} \right]-  \frac{o(t)}{t} \\[2mm]
&=& f''(x)-f''(x) =0.  
\end{eqnarray*}
\end{proof}


\section{A Trajectory Method Using Steklov regularization}

The trajectory approach we formulate is based on constructing a continuously differentiable path through points where
\begin{equation}  \label{grad}
\mu_x(x,t) = 0\,,\quad \forall t\in(0,t_0]\,.
\end{equation}
We interpret the variable $x$ as a function dependent on $t$, i.e., $x : [0,t_0] \to \dR$, mapping $t \mapsto x(t)$. By taking the total derivative of both sides of \eqref{grad} with respect to the independent variable $t$, we obtain
\begin{equation}  \label{grad1}
\mu_{xx}(x(t),t) \dot{x}(t) + \mu_{tx}(x(t),t) = 0\,,\hbox{ for a.e. } t\in(0,t_0]\,,
\end{equation}
where $\dot{x}$ stands for $dx/dt$. In particular, we note that, for $(x_0,t_0)=(x(t_0),t_0)$, we have by \eqref{grad} that $\mu_x(x_0,t_0) = 0$.  After re-arranging \eqref{grad1}, one obtains the initial value problem
\begin{equation}  \label{ODE_valley}
\dot{x}(t) = -\frac{\mu_{tx}(x(t),t)}{\mu_{xx}(x(t),t)}\,,\quad\mbox{ for a.e. } t\in(0,t_0]\,,\quad\mbox{with } x(t_0) = x_0\,,
\end{equation}
provided that $\mu_{xx}(x(t),t)\not=0$ a.e. in $(0,t_0]$.

\begin{remark} \rm 
Suppose that $x(\cdot)$ is a solution of the ODE in~\eqref{ODE_valley}.  Then Proposition~\ref{limits} implies that, if $\lim_{t\to 0^+}f''(x(t))\not= 0$, then $\lim_{t\to 0^+}\dot{x}(t) = 0$.
\end{remark}

\subsection{An algorithm for global optimization}

We motivate our first method as follows.  Assume that $f$ is as in Theorem~\ref{convexity}.  Let $x_0\in \dR$ and $t_0>0$ be such that 
\[
f'(x + t_0) - f'(x - t_0) > 0\,,\ \ \forall x\in\dR\,,\quad\mbox{and}\quad f(x_0 + t_0) - f(x_0 - t_0) = 0\,.
\]
From \eqref{mux-1} and \eqref{muxx}, the last two expressions imply that $\mu_{xx}(x,t_0) > 0$ and $\mu_x(x_0,t_0) = 0$, respectively. Using \eqref{grad} and \eqref{muxx}--\eqref{mutx} in the IVP~\eqref{ODE_valley}, we obtain 
\begin{equation}  \label{ODE_valley2}
\dot{x}(t) = -\frac{f'(x(t)+t) + f'(x(t)-t)}{f'(x(t)+t) - f'(x(t)-t)}\,,\quad\mbox{ for a.e. } t\in(0,t_0]\,,\quad\mbox{with } x(t_0) = x_0\,.
\end{equation}

Algorithm~\ref{algo1} below serves to find a global minimizer of $f$.

\begin{algorithm}  \label{algo1} \
\begin{description}
\vspace*{-3mm}
\item[Step \boldmath{$1$}] Choose the parameter $t_0>0$ large enough so that $\mu(\cdot,t_0)$ is convex. Find the (global) minimizer $x_0$ of $\mu(\cdot,t_0)$, i.e., solve $f(x_0 + t_0) - f(x_0 - t_0) = 0$ for $x_0$.
\item[Step \boldmath{$2$}] Solve the initial value problem in~\eqref{ODE_valley2}.
\item[Step \boldmath{$3$}] Report $\lim_{t\to 0^+} x(t)=:x^*$ as a global minimizer of $f$.
\end{description}
\end{algorithm}

Algorithm~\ref{algo1} is said to be {\em well-defined} for the function $f$ if there exist $x_0$ and $t_0>0$ such that Steps~1--3 of the algorithm can be carried out. This
entails, in particular,  that the solution of the IVP in Step~2 is obtained uniquely. Theorem~\ref{convexity} establishes assumptions on $f$ under which Step~1 can be carried out.

In the following lemma, we show that Algorithm~\ref{algo1} is {\em scale-shift invariant}; i.e., if Algorithm~\ref{algo1} is well-defined for the
function $f$, then it is also well-defined for any scale change and horizontal translation, of $f$.

\begin{lemma}[Scale-Shift Invariance] \label{shift_lemma}
Fix $\alpha>0$ and $a\in \dR$. Assume that Algorithm~\ref{algo1} is well-defined for $f$, and let $x_0$ and $t_0$ be as in Step 1 for $f$. Let $x^*$ be the global minimizer of $f$ generated by Step 3 of Algorithm~\ref{algo1} for $f$. Set $g(x) := f(\alpha\,x-a)$ and denote the Steklov function associated with $g$ by
\begin{equation}  \label{mf-func_shift}
\mut(x,t) := \frac{1}{2t}\,\int_{x-t}^{x+t} g(\tau)\,d\tau\,.
\end{equation}
Then Algorithm~\ref{algo1}, with $\mu$
replaced by $\mut$ is well-defined for $g$ and generates $z^*:=\dfrac{x^* + a}{\alpha}$, which is a
global minimizer  of $g$. In this case, Step 1 can be carried out with $s_0:=t_0/\alpha$, and $z_0:=\dfrac{x_0+a}{\alpha}$.
 
\end{lemma}
\begin{proof}
Using the definition of $g$ and \eqref{mf-func_shift} we can write 
\begin{equation} \label{mu_shift} 
\begin{array}{rcl}
\mut(x,t)& =& \ds\frac{1}{2t}\,\int_{x-t}^{x+t} f(\alpha\,\tau - a)\,d\tau\, =
\frac{1}{2t}\,\int_{\alpha(x-t)-a}^{\alpha(x+t)-a} f(\eta)\,d\eta\,\\
&&\\
& = &  \ds\frac{\alpha}{2(\alpha\,t)}\,\int_{\alpha\,x-a-(\alpha t)}^{\alpha\,x-a+(\alpha t)} f(\eta)\,d\eta\,=
\alpha\,\mu(\alpha\,x-a,\alpha\,t)\,,
\end{array}
\end{equation}
through a change of the dummy integration variable, $\eta = \alpha\,\tau - a$, and the definition in \eqref{mf-func}.  Then, by taking partial derivatives of $\mut$, where we employ the chain rule on the right-most term of the second line in \eqref{mu_shift}, we get
\begin{equation} \label{mu_shift_derivs} 
\begin{array}{cc}
\mut_x(x,t) = \alpha^2\mu_x(\alpha\,x-a,\alpha\,t),&\mut_{xx}(x,t) =
\alpha^3\mu_{xx}(\alpha\,x-a,\alpha\,t)\\
&\\
\mut_{tx}(x,t) = \alpha^3\mu_{tx}(\alpha\,x-a,\alpha\,t)\,.&
\end{array}
\end{equation}
Since Algorithm~\ref{algo1} is well-defined for $f$, Steps~1 and 2 of the
algorithm can be executed, generating a global minimizer $x_0$ of  $\mu(\cdot,t_0)$ in Step 1, a unique solution $x(\cdot)$ to the IVP
in~\eqref{ODE_valley} in Step 2, where $\mu(\cdot,t_0)$ is convex and
\begin{equation}\label{q:14}
\mu_{x}(x_0,t_0) = 0\,.
\end{equation}
In Step~3, a global minimizer of $f$ is obtained as $\lim_{t\to 0^+} x(t) = x^*$. Take $z_0 := (x_0+a)/\alpha$ and $s_0:=t_0/\alpha$.  We show now that {Step~1} is well defined for $g$, for $s_0$ and $z_0$ in place of $t_0$ and $x_0$, respectively. Indeed, by Step 1 for $f$ we know that $\mu(\cdot,t_0)$ is convex. Hence, the composition of $\mu(\cdot,t_0)$ with the linear function $L(x)=\alpha\,x-a$ is also convex. Namely, the function $\mu(\alpha\,(\cdot)-a,t_0)$ is convex, and hence any positive multiple of it is convex. Therefore, by \eqref{mu_shift} we deduce that $\mut(\cdot,t_0)$ is convex. 
The first  equality in \eqref{mu_shift_derivs},  combined with \eqref{q:14} and the definitions of $z_0$ and $t_0$ give
\[
0=\mu_{x}(x_0,t_0)=\mu_{x}(\alpha z_0-a,\alpha\, s_0)=\frac{1}{\alpha^2}\,\mut_x(z_0,s_0)\,,
\]
so that $z_0$ is a global minimizer of $\mut(\cdot,s_0)$. This shows that {Step~1} is well defined for $g$. We proceed now to show that Step 2 is well defined for $g$. Take  $x(\cdot)$ to be the unique solution of the IVP in~\eqref{ODE_valley} obtained in Step 2 for $f$, and define $z(t) :=(x(\alpha\,t)+a)/\alpha$, for all $t\in(0,s_0]$. We claim that $z(\cdot)$ solves the following IVP:
\begin{equation}  \label{ODE_valley3}
\dot{z}(t) = -\frac{\mut_{tx}(z(t),t)}{\mut_{xx}(z(t),t)}\,,\quad\mbox{ for
  a.e. } t\in(0,s_0]\,,\quad z(s_0) = z_0\,,
\end{equation}
which is the IVP in~\eqref{ODE_valley} with $\mu$ replaced by $\mut$ and $x_0$ replaced by $z_0$. Indeed, take $t\in (0,s_0]=(0,t_0/\alpha]$. Then $\alpha\,t\in (0,t_0]$ and by \eqref{ODE_valley} we can write
\[
\dot{z}(t)=\dot{x}(\alpha\,t)=-\frac{\mu_{tx}(x(\alpha\,t),\alpha\,t)}{\mu_{xx}(x(\alpha\,t),\alpha\,t)}=
-\frac{\mu_{tx}(\alpha\,z(t)-a,\alpha\,t)}{\mu_{xx}(\alpha\,z(t)-a,\alpha\,t)}=-\frac{\mut_{tx}(z(t),t)}{\mut_{xx}(z(t),t)},
\]
where we have used the definition of $z$ in the first equality, the definition of $x$ as solution of \eqref{ODE_valley} in the second equality.  We have used the second and third equalities of \eqref{mu_shift_derivs} in the third equality above. The fact that $z(s_0)=z_0$ follows directly from the definition of $z$ and the fact that $x(t_0)=x_0$. Therefore, $z$ solves \eqref{ODE_valley3} and hence Step 2 is well defined for $g$. To check that the same holds for Step 3, take
$x^*=\lim_{t\to 0^+} x(t)$ to be the global minimizer of $f$ generated in Step 3 for $f$. The solution of the IVP in \eqref{ODE_valley3} will now result in
\[
\lim_{t\to 0^+} z(t) = \lim_{t\to 0^+} \frac{x(\alpha\,t) + a}{\alpha} = \frac{x^* + a}{\alpha}\,.
\]
Since $x^*$ is a global minimizer of $f$, we have
\[
g\left(\frac{x^* + a}{\alpha}\right)=f(x^*)\le f(\alpha\,x-a)=g(x),\quad, \forall\, x\in \dR\,,
\]
so $(x^* + a)/\alpha$ is a global minimizer of $g$. Hence, Step 3 is well defined for $g$ and the proof is complete.
\end{proof}

\section{Quartic Polynomials}
\label{sec:quartic}

In this section, we consider the special case of monic {\em depressed} quartic polynomials, namely,
\begin{equation}  \label{quartic}
f(x) = x^4 + a_2\,x^2 + a_1\,x + a_0\,,
\end{equation}
where $a_0,a_1$ and $a_2$ are real constants such that $a_2<0$ and $a_1\neq0$. Note that the depressed form is general enough. Indeed, given an arbitrary quartic polynomial, $g(y) = y^4 + b_3\,y^3 + b_2\,y^2 + b_1\,y + b_0$, the substitution $y = x-(b_3/4)$ reduces $g$ to a depressed form. As for the assumption $a_2<0$, note that, if $a_2\ge0$ then $f''(x) = 12\,x^2 + 2\,a_2\ge0$ which yields that $f(\cdot)$ is convex. In this case, there is no need to apply Algorithm~\ref{algo1} to find a global minimum of $f(\cdot)$.  Assumption $a_1\neq0$ is posed since if $a_1=0$ then $f(\cdot)$ has two global minimizers simply given by the set $\{-\sqrt{-a_2/2},\sqrt{-a_2/2}\}$.  Hence, the non-trivial case is when  $a_2<0$ and $a_1\neq0$.

\subsection{Properties of monic quartic polynomials}

A quartic polynomial can have at most two local minima. The following lemma helps distinguish which of these two is the global minimum.

\begin{lemma}[Curvature] \label{curvature} 
Let $f$ be a monic quartic polynomial. Assume that $f'(x_1) = f'(x_2) = 0$ with $x_1\not=x_2$.  The following properties hold.
\begin{itemize}
\item[(i)] $f(x_1) < f(x_2)$ if, and only if, $f''(x_1) > f''(x_2)$, in particular, $|x_1| > |x_2|$.
\item[(ii)] $f(x_1) = f(x_2)$ if, and only if, $f''(x_1) = f''(x_2)$.
\end{itemize}
\end{lemma}
\begin{proof} 
The proof of parts (i) and (ii) is done in two steps. \\
{\sc Step 1:} In this step, we show that it is enough to prove the lemma for a depressed quartic polynomial. We prove the claim for part (i). The claim for part (ii) is proved in an identical way. Assume that part (i) of the lemma is true for depressed monic quartic polynomials and that we have a quartic polynomial $h(x)=x^4 +c_3x^3 + c_2 x^2 +c_1 x +c_0$ with $c_3\not=0$. Assume that $h'(x_1)= h'(x_2) = 0$ As noted above, the ``shifted'' polynomial $f(x):=h(x-(c_3/4))$ is (monic and) depressed. Using the chain rule we have
\[
h'(x_1)=f'(x_1+(c_3/4)) = h'(x_2) = f'(x_2+(c_3/4)) = 0\,.
\]
Since part (i) of the lemma is true for $f$ we have
\[
h(x_1) = f(x_1+(c_3/4))<f(x_2+(c_3/4)) = h(x_2)
\]
if and only if
\[
h''(x_1)=f''(x_1+(c_3/4))>f''(x_2+(c_3/4))=h''(x_2)\,,
\]
and hence part~(i) of the lemma holds for $h$. As mentioned before, the proof of the fact that part~(ii) of the lemma holds for $h$ follows identical steps. Therefore, it is enough to prove the lemma for depressed quartic polynomials.\\ {\sc Step 2:}  In this step, we show that, if $f$ is a depressed quartic polynomial such that $f'(x)=f'(y)=0$ , with $x\not=y$, then we have
\begin{equation}  \label{eq:1a}
12\,\frac{[f(x)-f(y)]}{(x-y)^2}=(f''(y)-f''(x)).
\end{equation}
Note that parts (i) and (ii) of the lemma follow directly from \eqref{eq:1a}. This is straightforward for part (ii). As for part (i), if \eqref{eq:1a} holds, the assumption on $x_1$ and $x_2$ implies that for $x:=x_1$ and $y:=x_2$ we have
\[
\sgn\left[f(x_1)-f(x_2)\right] = \sgn\left[f''(x_2)-f''(x_1)\right] = -\sgn\left[f''(x_1)-f''(x_2)\right],  
\]
which is the statement of part (i) of the lemma, also observing that $f''(x_1)>f''(x_2)$ if and only if $12x_1^2 + 2a_2 > 12x_2^2 + 2a_2$, i.e., $|x_1| > |x_2|$. Hence, we proceed to prove \eqref{eq:1a} when $f'(x)=f'(y)=0$ and $f$ is a depressed quartic polynomial. The assumption on $x$ and $y$ and the Taylor development of $f$ gives
\begin{eqnarray*}
 f(x)-f(y)&=&[f''(y)/2](x-y)^2 + [f'''(y)/6] (x-y)^3 + (x-y)^4\,, \\
 f(y)-f(x)&=&[f''(x)/2](y-x)^2 + [f'''(x)/6] (y-x)^3 + (y-x)^4\,,
\end{eqnarray*}
By subtracting side-by-side the second equality from the first one, and re-arranging the resulting expression we obtain
\begin{equation}
  \label{eq:2a}
 \begin{array}[h]{rcl}
\displaystyle 2\frac{[f(x)-f(y)]}{(x-y)^2}&=&\displaystyle  \frac{[f''(y)-f''(x)]}{2}+ \frac{[f'''(y)+f'''(x)]}{6} (x-y).
\end{array}
 \end{equation}

 By direct calculation, the rightmost term in \eqref{eq:2a} can be
 written as follows
\begin{eqnarray*}
\frac{[f'''(y)+f'''(x)]}{6} (x-y) &=& \frac{[24y+24x]}{6} (x-y)=4(x^2-y^2) \\[2mm]
&=& \left[(12x^2+2a_2)-(12y^2+2a_2) \right] / 3 \\[1mm]
&=& [f''(x)-f''(y)] / 3 \\[1mm]
&=& -[f''(y)-f''(x)] / 3\,.
\end{eqnarray*}
Using this in \eqref{eq:2a} yields
\begin{equation}
  \label{eq:3}
 \begin{array}[h]{rcl}
\displaystyle 2\,\frac{[f(x)-f(y)]}{(x-y)^2}&=&\displaystyle
\frac{[f''(y)-f''(x)]}{2} -\frac{[f''(y)-f''(x)]}{3} =\frac{[f''(y)-f''(x)]}{6},
\end{array}
 \end{equation}
which is  \eqref{eq:1a}. The proof is complete.
\end{proof}

\begin{remark}\rm
The previous lemma is not valid for higher degree polynomials. The function $f(x)=x^6 -\dfrac{8}{5}\,x^5+\dfrac{2}{3}\,x^3$, with the local extrema $x_1=0$ and $x_2=1$, furnishes a counterexample.
\end{remark}

\begin{lemma}[Sign of a Minimizer]  \label{global_sign}
Consider a monic depressed quartic polynomial $f$, with $a_1\neq0$ and $a_2<0$. If $x_1$ and $x_2$ are the local minimizers of $f$, then $\sgn(x_1) = -\sgn(x_2)$. Suppose that $x^*$ is the global minimizer of $f$. Then $|x^*| > \sqrt{-a_2/6}$ and, in particular, $\sgn(x^*) = -\sgn(a_1)$. 
\end{lemma} 
\begin{proof}
Suppose that $x_1$ and $x_2$ are the local minimizers of $f$.  Note that $f''(x) = 12\,x^2 + 2\,a_2$ is an even function, i.e., $f''(-x) = f''(x)$.  Since $a_2<0$, we have $f''(x) = 0$ when $x = \sqrt{-a_2/6} =: \widetilde{x}$ and $x = -\widetilde{x}$.  Then, since $f''(x)>0$ for $x<-\widetilde{x}$ and $x>\widetilde{x}$, one of the local minima is placed to the left of $-\widetilde{x}$ and the other to the right of $\widetilde{x}$, i.e., $x_1<-\widetilde{x}<0$ and $x_2>\widetilde{x}>0$.  So $\sgn(x_1) = -\sgn(x_2)$ and $|x^*| > \sqrt{-a_2/6}$. Now, we can write 
\begin{eqnarray}
f'(x_1) &=& f'(-\widetilde{x}) + \int_{-\widetilde{x}}^{x_1} f''(y)\,dy \nonumber \\
&=& f'(-\widetilde{x}) - \int_{\widetilde{x}}^{-x_1} f''(y)\,dy = 0 \label{int1}
\end{eqnarray}
where, in \eqref{int1}, a change of variables and the fact that $f''$ is even have been used.  We can also write
\begin{equation}  \label{int2}
f'(x_2) = f'(\widetilde{x}) + \int_{\widetilde{x}}^{x_2} f''(y)\,dy = 0\,.
\end{equation}
Adding Equations~\eqref{int1}-\eqref{int2} side by side and using
$f'(\widetilde{x}) + f'(-\widetilde{x}) = 2\,a_1$, one gets
\begin{equation}  \label{int1&2}
2\,a_1 + \int_{\widetilde{x}}^{x_2} f''(y)\,dy - \int_{\widetilde{x}}^{-x_1} f''(y)\,dy = 0\,.
\end{equation}
To complete the proof, we consider two cases: $a_1<0$ and $a_1>0$. If $a_1<0$, from \eqref{int1&2} we have,
\[
-2\,a_1=\int_{\widetilde{x}}^{x_2} f''(y)\,dy - \int_{\widetilde{x}}^{-x_1} f''(y)\,dy > 0\,,
\]
which implies that $x_2>-x_1>0$, since $f''(y)>0$ over both integration intervals. Since $f'''(x)=24x>0$ for $x>0$, $f''$ is increasing in $(0,\infty)$ so $f''(x_2)>f''(-x_1)=f''(x_1)$.  Then, by Lemma~\ref{curvature}, $f(x_2)<f(x_1)$ and hence $x^*=x_2>0$ is the global minimizer.  Therefore, $\sgn(x^*) = 1=-\sgn(a_1)$.

Suppose now that $a_1>0$.  Through similar steps, we get $-x_1>x_2>0$, or $x_1<-x_2<0$. Since $f'''(x)=24x<0$ for $x<0$, $f''$ is decreasing in $(-\infty,0)$ so $f''(x_1)>f''(-x_2)=f''(x_2)$.  Then, by Lemma~\ref{curvature}, $f(x_1)<f(x_2)$ and hence $x^*=x_1<0$ is the global minimizer.  Therefore, $\sgn(x^*) = -1=-\sgn(a_1)$, completing the proof.
\end{proof}

\subsection{An algorithm for global minimization of quartic polynomials}

In this section we consider the specific case of applying Algorithm~\ref{algo1} to quartic polynomials.

\begin{proposition}  \label{lem:mu_univar}
If $f(x)$ is a univariate quartic monic polynomial, then $\mu(x,t)$, as in \eqref{mf-func}, can be written as
\begin{equation}  \label{eq:mu_univar}
\mu(x,t) = f(x) + \frac{t^2}{6}\,f''(x) + \frac{t^4}{5}\,.
\end{equation}
\end{proposition}
\begin{proof}
Let $f(x) := x^4 + a_3\,x^3 + a_2\,x^2 + a_1\,x + a_0$, where $a_0,a_1,a_2$ and $a_3$ are real numbers.  Substitution of $f$ into \eqref{mf-func}, followed by straightforward integration, expanding and rearranging, yield \eqref{eq:mu_univar}.
\end{proof}

Using \eqref{eq:mu_univar}, $\mu(x,t)$ and its derivatives can now be re-written for monic depressed polynomials as follows.
\begin{eqnarray}
\mu(x,t) &=& f(x) + \frac{t^2}{6}\,f''(x) + \frac{t^4}{5} 
= x^4 + (a_2 + 2\,t^2)\,x^2 + a_1\,x + a_0 + \frac{t^2}{3} +
             \frac{t^4}{5}\,,  \label{mu}  \\[1mm]
\mu_x(x,t) &=& f'(x) + \frac{t^2}{6}\,f'''(x) = 4\,x^3 + 2\,(a_2 +
               2\,t^2)\,x + a_1\,,  \\[1mm] 
\mu_{xx}(x,t) &=& f''(x) + 4\,t^2 = 12\,x^2 + 2\,(a_2 + 2\,t^2)\,,
\label{muxx_quartic}  \\[1mm] 
\mu_{tx}(x,t) &=& \frac{t}{3}\,f'''(x) = 8\,t\,x\,.  \label{mutx_quartic}
\end{eqnarray}

In Step 1 of Algorithm~\ref{algo1}, we need to find (i) some $t_0>0$ such that $\mu(\cdot,t_0)$ is convex, and (ii)~the point $x_0$ which is the global minimizer of $\mu(\cdot,t_0)$. The next lemma finds these for the particular case of quartic polynomials.

\begin{lemma}[Convexification of Quartic Polynomials] \label{t0_x0} Given any monic depressed quartic polynomial $f(x)$ with $a_1\neq0$ and $a_2<0$, $\mu(\cdot,t_0)$ is convex if
\[
t_0:=\sqrt{-a_2/2}\,.
\]
The unique (global) minimizer of $\mu(\cdot,t_0)$ is
\[
x_0 := -\sqrt[3]{a_1/4}\,.
\]
\end{lemma}
\begin{proof}

For convexity of $\mu(\cdot,t_0)$, we need to have
\[
\mu_{xx}(x,t_0) = 12\,x^2 + 2\,(a_2 + 2\,t_0^2) \ge 0\,,\quad\forall x\in\dR,
\]
which immediately follows if $t_0=\sqrt{-a_2/2}$.  The (global) minimizer $x_0$ of $\mu(\cdot,t_0)$ would then be found by solving 
\begin{equation}  \label{initialeq}
\mu_x(x_0,t_0) = 4\,x_0^3 + 2\,(a_2 + 2\,t_0^2)\,x_0 + a_1 = 0\,,
\end{equation}
or, with $t_0=\sqrt{-a_2/2}$,
\[
4\,x_0^3 + a_1 = 0\,,
\]
for $x_0$. Since $a_1\not=0$, we have $x_0\not=0$. This implies that $\mu_{xx}(x_0,t_0)>0$, and hence $\mu(\cdot,t_0)$ is strictly convex around $x_0$, which implies that $x_0$ is the unique minimizer. This completes the proof.
\end{proof}

For the special case of monic depressed quartic polynomials,
Algorithm~\ref{algo1} reduces to the following, using Lemma~\ref{t0_x0},
\eqref{muxx_quartic} and \eqref{mutx_quartic}.

\begin{algorithm} \label{algo2} \
\begin{description}
\vspace*{-3mm}
\item[Step \boldmath{$1$}] Given the quartic polynomial in \eqref{quartic}, let $t_0 = \sqrt{-a_2/2}$\ \ and\ \ 
  $x_0 = -\sqrt[3]{a_1/4}$\,.
\item[Step \boldmath{$2$}] Solve the initial value problem
\begin{equation}  \label{ODE_valley_quartic}
\dot{x}(t) = -\frac{\ds 4\,t\,x(t)}{6\,x^2(t) + 2\,t^2 + a_2}\,,
\quad\mbox{ for a.e. } t\in[0,t_0]\,,\quad\mbox{with }  
x(t_0) = x_0\,.
\end{equation}
\item[Step \boldmath{$3$}] Report $x(0)$ as the global minimizer of
  $f(x)$.
\end{description}
\end{algorithm}

By \eqref{grad} and \eqref{ODE_valley}, IVP \eqref{ODE_valley_quartic} can be derived under the assumption that
$\mu_x(x(t),t) = 0$ and $\mu_{xx}(x(t),t) \not= 0$ for a.e. $t\in [0,t_0]$. Hence, it is worth investigating if, at all, the denominator of the right-hand side of the ODE in~\eqref{ODE_valley_quartic} vanishes, i.e., $\mu_{xx}(x(t),t) = 0$ for some $t\in[0,t_0]$.  Since a solution $x(t)$ of the ODE in~\eqref{ODE_valley_quartic} satisfies $\mu_x(x(t),t) = 0$, the pathological situation happens at points $(x,t)$ which satisfy the equations $\mu_x(x,t) = 0$ and $\mu_{xx}(x,t) = 0$ simultaneously. We investigate this situation in the following lemma.

\newpage
\begin{lemma}[Flatness] \label{flatness} 
Let  $f$ be a monic depressed quartic polynomial. Consider solutions $(x,t)\in \dR\times [0,\infty)$ of the system 
\begin{equation} \label{system}
\mu_x(x,t) = 0\qquad\mbox{and}\qquad \mu_{xx}(x,t)=0\,.
\end{equation}
\begin{itemize}
\item[(a)] If $a_2> -3\,a_1^{2/3}/2$ then the system in~\eqref{system} has no solution. 
\item[(b)] If $a_2\le  -3\,a_1^{2/3}/2$ then the system in~\eqref{system} has a unique solution $(\widehat{x},\widehat{t})\in  \dR\times [0,\infty)$ such that
\begin{eqnarray}  
\widehat{x} &:=& \frac{1}{2}\,\sqrt[3]{a_1}\,,  \label{x_valley_end}  \\
\widehat{t} &:=& \frac{1}{2}\,\sqrt{-\left(3\,a_1^{2/3} + 2\,a_2\right)}\,.  \label{t_valley_end} 
\end{eqnarray}
\end{itemize} 
\end{lemma} 
\begin{proof}
Start with
\begin{eqnarray}
\mu_x(\widehat{x},\widehat{t}) &=&
4\,\widehat{x}^3 + 2\,(a_2 + 2\,\widehat{t}^2)\,\widehat{x} + a_1 = 0\,, \label{mu_x} \\ 
\mu_{xx}(\widehat{x},\widehat{t}) &=& 12\,\widehat{x}^2 + 2\,(a_2 + 2\,\widehat{t}^2) = 0\,.  \label{mu_xx} 
\end{eqnarray}
Using \eqref{mu_xx} in \eqref{mu_x} gives $ \widehat{x}=\frac{1}{2}\,\sqrt[3]{a_1}$. Using this value of $\widehat{x}$ in \eqref{mu_xx} gives
\begin{equation}\label{q1}
0= 6\,\left(\frac{1}{2}\,\sqrt[3]{a_1}\right)^2 + (a_2 + 2\,\widehat{t}^2) = \frac{3}{2} a_1^{2/3}+a_2+2 \,\widehat{t}^2.
\end{equation}
To prove part (a), note that $a_2> -3\,a_1^{2/3}/2$ if and only if 
\[
0= \frac{3}{2} a_1^{2/3}+a_2+2 \,\widehat{t}^2 > 2 \,\widehat{t}^2,
\]
which entails a contradiction and therefore implies that the system has no solution. This proves (a). On the other hand, $a_2\le -3\,a_1^{2/3}/2$ if and only if there is a unique nonnegative solution of \eqref{q1}, given by \eqref{t_valley_end}. This proves part (b). 
\end{proof}

We will consider in our analysis a notion which is weaker than convexity, called {\em quasi-convexity}.

\begin{definition} \rm
The function $f:\mathbb{R}^n\to \mathbb{R}$ is said to be {\em quasi-convex} when all its level sets are convex, i.e., when for every $\alpha\in \mathbb{R}$ we have that the set $\{x\in \mathbb{R}^n\::\: f(x)\le \alpha\}$ is convex.
\end{definition}

\begin{definition} \rm
Let $I$ be a (possibly infinite) interval in $\mathbb{R}$. Recall that $f:\mathbb{R}\to \mathbb{R}$ is {\em non-increasing in} $I$ if for all $x,y\in I$ such that $x<y$ we have $f(x)\ge f(y)$. Similarly, $f$ is {\em non-decreasing in} $I$ if for all $x,y\in I$ such that $x<y$ we have $f(x)\le f(y)$.
\end{definition}

The following result is a trivial re-statement of Theorem 4.9.11 in \cite{SW}.
\begin{theorem}[Quasi-convexity] \label{thm:Stoer} 
A function $f:\mathbb{R}\to \mathbb{R}$ is quasi-convex if and only if there exists $m\in [-\infty,+\infty]$ such that $f$ is
non-increasing in $(-\infty,m]$ and non-decreasing in  $[m, +\infty)$.
\end{theorem}

\begin{remark}  \rm
When $m$ is infinite in Theorem \ref{thm:Stoer}, one of the intervals is empty and this covers the case in which the function is everywhere non-increasing or everywhere non-decreasing. Functions as in the statement of Theorem \ref{thm:Stoer} are sometimes called {\em unimodal}.
\end{remark}

\begin{lemma}[Quasi-convexity of a Quartic Polynomial] \label{f: quasi-convex} 
Let $h$ be a monic depressed quartic polynomial given by $h(x):=x^4 +b_2x^2+b_1 x +b_0$. Consider $\Delta := -16\,[8\,b_2^3 + 27\,b_1^2]$, i.e., $\Delta$ is the {\em discriminant} of $h'(x) = 4\,x^3 + 2\,b_2\,x + b_1$. Then $h$ is quasi-convex if and only if  $\Delta\le 0$. In this situation, we have $b_2 \ge -3\,b_1^{2/3}/2$.
\end{lemma}
\begin{proof}
From algebra of cubic equations we have that
\begin{eqnarray*}
(I) && \Delta > 0 \Longrightarrow h' \hbox{ has three distinct real roots},\\
(II) && \Delta = 0 \Longrightarrow h' \hbox{ has a multiple root and all its roots are real},\\
(III) && \Delta < 0 \Longrightarrow h' \hbox{ has one real root and two non-real complex conjugate roots},
\end{eqnarray*}
Case (I) implies that $h$ has a local maximum, and two local minima. This cannot hold for a quasi-convex function in view of Theorem \ref{thm:Stoer}. Hence it is enough to show that the other two cases imply that $h$ satisfies the unimodality property described in Theorem \ref{thm:Stoer}.  This is clear in case (III), since by coercivity the unique real root of $h'$ must be a global minimum. So in case (III) $h$ is quasi-convex by Theorem \ref{thm:Stoer}. In Case (II), we have two possibilities: either all three roots coincide (i.e., we have a triple root of $h'$) or one of the real roots is double. The case in which we have a triple root of $h'$ implies that $h'(x)=4(x-x_0)^3$ and hence is strictly increasing. So $h$ is convex, and hence quasi-convex.  We are left only with the case of a double real root $x_0$ and a simple real root $x_1$. In this case, $h'(x)=4(x-x_1)(x-x_0)^2$ and it is clear that $h'(x)\le 0$ for $x\le x_1$ and $h'(x)\ge 0$ for $x\ge x_1$. This implies directly (using mean value theorem) that $h$ verifies the unimodality property given in Theorem \ref{thm:Stoer} with $m=x_1$, and hence $h$ is quasi-convex. We have shown that $\Delta \le 0$ implies $h$ quasi-convex. Conversely, assume that $h$ is quasi-convex, and let $m$ be as in Theorem \ref{thm:Stoer}. Since $h$ is coercive, we must have $m\in\mathbb{R}$. Because $h$ is a polynomial, $h$ cannot be constant in any interval, so $m$ must be the only global minimum of $h$. This implies that $h'$ cannot have three different roots, so we cannot be in Case (I) and hence we must have $\Delta \le 0$. The last statement of the lemma follows directly from the expression of the discriminant.
\end{proof}

\begin{proposition} \label{mu:quasi-convex} 
Suppose that $f$ is a monic depressed quartic polynomial.  Then $\mu(\cdot,t)$ is quasi-convex if, and only if,
\begin{equation}  \label{t_condition}
t \ge \frac{1}{2}\,\sqrt{  \max\left\{0,-\left(3\,a_1^{2/3} + 2\,a_2\right)\right\}}\,.
\end{equation}
\end{proposition}
\begin{proof}
By \eqref{mu}, $\mu(\cdot,t)$ is a monic quartic depressed polynomial with $\mu_x(x,t) = 4\,x^3 + 2\,(a_2 + 2\,t^2)\,x + a_1$.  By Lemma~\ref{f: quasi-convex} applied to $\mu(\cdot,t)$, we have that $\mu(\cdot,t)$ is quasi-convex if, and only if, the discriminant of $\mu_x(\cdot,t)$ is non-positive, i.e.,
\[
\Delta = -16 \left[8\,(a_2 + 2\,t^2)^3 + 27\,a_1^2\right] \le 0\,,
\]
a re-arrangement of which yields \eqref{t_condition}.
\end{proof}

\begin{remark}\rm
Lemmas  \ref{flatness} and \ref{f: quasi-convex} imply that the right-hand-side of the ODE in~\eqref{ODE_valley_quartic} cannot be discontinuous in the interior of $[0,t_0]$ when $f$ is a quasi-convex quartic polynomial. Indeed, assume that the right-hand-side of the ODE in \eqref{ODE_valley_quartic} is discontinuous in the interior of $[0,t_0]$ and $f$ is quasi-convex. By Lemma \ref{flatness}, this implies that the system \eqref{x_valley_end}--\eqref{t_valley_end} has a solution in $\dR\times (0,t_0)$. By part (b) of the lemma this yields $(3\,a_1^{2/3} + 2\,a_2)\le 0$. On the other hand, by Lemma \ref{f: quasi-convex}, $f$ is quasi-convex if and only if $(3\,a_1^{2/3} + 2\,a_2)\ge 0$. This yields $(3\,a_1^{2/3} + 2\,a_2)= 0$. Using \eqref{t_valley_end} gives $\widehat{t}=0$ as the unique solution of the system. Since the nonnegative solution is unique, this implies that there is no solution for $t$ in the interior of $(0,t_0)$, and hence the denominator in the right-hand side of the ODE in  \eqref{ODE_valley_quartic} cannot vanish for $t\in (0,t_0)$.
\end{remark}

\subsection{Well-definedness of Algorithm~\ref{algo2}}

\begin{lemma}[Solutions of ODEs] \label{lem:solvability} 
Consider a monic depressed quartic polynomial $f$, with $a_2<0$ and $a_1\neq0$. Let $t_0 = \sqrt{-a_2/2}$\ \ and\ \ $x_0 = -\sqrt[3]{a_1/4}$\,.  The following hold.
\begin{enumerate}
\item[(a)] There exists $r>0$ such that there is a unique solution $x(\cdot)$ of \eqref{ODE_valley_quartic} in $(t_0-r,t_0+r)$.
\item[(b)] There exists a maximal interval to the left of $t_0$, say $(m_0,t_0]$, such that there exists a solution of
  \eqref{ODE_valley_quartic} in $(m_0,t_0]$.
\item[(c)] Either $m_0 = -\infty$, or $m_0\in \mathbb{R}$ and in this case we must have $\mu_{xx}(x(m_0),m_0) = 0$.
\end{enumerate}
\end{lemma}
\begin{proof}
Note that $x_0^2>0$ because $a_1\not=0$.
Part~(a) follows from the classical Picard-Lindel\"of existence and uniqueness theorem (see \cite{Arnold1978}), since the denominator $6\,x^2(t_0) + 2\,t_0^2 + a_2 = 6\,x_0^2 >0$. By \eqref{muxx_quartic}, this implies that $\mu_{xx}(x(t_0),t_0) > 0$, and so the right-hand side of the ODE in \eqref{ODE_valley_quartic} is Lipschitz continuous in $x$ and continuous in $t$ in a neighbourhood of $t_0$. Part~(b) is the classical result on maximal extension of solutions of ODEs. The option $m_0 = -\infty$ of part~(c) corresponds to the case in which the right-hand side remains Lipschitz continuous in $x$ for all $t<t_0$. The remaining option happens when the denominator
\begin{equation}  \label{denom}
q(t) := 6\,x^2(t) + 2\,t^2 + a_2
\end{equation}
vanishes at $t = m_0$, i.e., when $\mu_{xx}(x(m_0),m_0) = 0$. This completes the proof.
\end{proof}

The following lemma re-formulates the initial value problem in \eqref{ODE_valley_quartic}.

\begin{lemma}[Trajectory Along a Valley] \label{lem:1} 
Consider a monic depressed quartic polynomial, with $a_2<0$ and $a_1\neq0$.  Let $t_0 = \sqrt{-a_2/2}$\ \ and\ \ $x_0 = -\sqrt[3]{a_1/4}$\,. With the notation of Lemma \ref{lem:solvability}, let $x(\cdot)$ be the maximally extended solution of~\eqref{ODE_valley_quartic}, and $(m_0,t_0]$ the corresponding maximal interval.  Then, we have that
\begin{equation}\label{trajectory condition}
\mu_x(x(t),t) = 0\,,\ \ \mu_{xx}(x(t),t) > 0\,,\ \forall t\in
[m_0,t_0]\,.
\end{equation}
\end{lemma}
\begin{proof}
We show first that $\mu_{xx}(x(t),t) > 0\,,\ \forall t\in
[m_0,t_0]\,$. Indeed, the choices of $x_0$ and $t_0$, together with \eqref{muxx_quartic} give $\mu_{xx}(x(t_0),t_0) > 0$. The definition of $m_0$ states that  IVP \eqref{ODE_valley_quartic} is solvable over $(m_0,t_0]$. By Lemma \ref{lem:solvability}(a), this implies that the right-hand side of the ODE is continuous on $(m_0,t_0]$. In other words, the denominator of the right-hand side of the ODE is not zero and so it does not change sign on $(m_0,t_0]$. This readily gives  
\begin{equation}  \label{pos_phi_xx}
\mu_{xx}(x(t),t) > 0\,,
\end{equation}
for all $t\in(m_0,t_0]$, as wanted. To complete the proof, recall that the ODE in \eqref{ODE_valley_quartic} is the ODE in \eqref{ODE_valley} written
for a quartic polynomial. Then, for all $t\in(m_0,t_0]$, the ODE in \eqref{ODE_valley} is equal to the expression in \eqref{grad1}, which is
\[
\dot{x}(t)\,\mu_{xx}(x(t),t) + \mu_{tx}(x(t),t) = 0\,.
\]
The above expression can in turn be expressed as 
\begin{equation} \label{eq:1}
  \frac{d}{dt}\,\mu_x(x(t),t) = 0\,.
\end{equation}
From the first step of Algorithm~\ref{algo1},
\begin{equation}  \label{eq:2}
  \mu_x(x(t_0),t_0) = 0\,.
\end{equation}
Equalities \eqref{eq:1} and \eqref{eq:2} imply that 
\begin{equation}
  \label{eq:4}
  \mu_x(x(t),t) = 0\,,
\end{equation}
for all $t\in(m_0,t_0]$. Equality \eqref{eq:4} holds at $t=m_0$ by continuity of $\mu_x$ and $x(\cdot)$. This completes the proof of the lemma.
\end{proof}

Our next step is to show that the solution $x(\cdot)$ of the initial value problem \eqref{ODE_valley_quartic} has the same sign as that of $x_0$ over its maximal domain of definition. For proving this, we need the following auxiliary result.

\begin{lemma} \label{sign-ODE}
Fix $a\in \dR\cup\{-\infty\}$ and let $b>a$. Assume that the function $x(\cdot):(a,b] \to \dR$ is continuously differentiable and satisfies
\begin{equation}\label{sign}
\sgn(\dot{x}(t))=-\sgn(x(t)),\,\quad\forall\, t\in (a,b].
\end{equation}
Assume that $x(b)\not=0$. Then, $\sgn(x(t))=\sgn(x(b))$ for all $t\in (a,b]$. If $a\in \dR$, then $\sgn(x(a))=\sgn(x(b))$.
\end{lemma}
\begin{proof}
Without loss of generality, assume that $x(b)<0$. The case $x(b)> 0$ is handled similarly, {\em mutatis mutandis}. We show first that $x(t)<0$ for all $t\in (a,b]$. Suppose that, on the contrary, there exists $\tilde t \in (a,b]$ such that  $x(\tilde t\,)\ge 0$. This implies that the set
\[
S:=\{ t \in (a,b]\::\: x(t)\ge 0\},
\]
is not empty. Since $S$ is bounded above by $b$, there exists $t_1:=\sup(S)$. Since $x(\cdot)$ is continuous, we have that $t_1\in S$, or, equivalently, $x(t_1)\ge 0$. In particular, $t_1<b$ because $x(b)<0$ and $x(t_1)\ge 0$. The definition of $t_1$ implies that $x(t)<0$ for all $t\in (t_1,b]$.  Using the continuity of $x$ we deduce that $x(t_1)\le 0$. Altogether, we must have $x(t_1)=0$.  The Mean Value Theorem gives, for some $\theta\in (t_1,b)$:
\[
0=x(t_1)=x(b) +\dot{x}(\theta)(t_1-b)<x(b)<0,
\]
where we used \eqref{sign} for $t=\theta$ in the first inequality, and the fact that $x(\theta)<0$. The above expression entails a contradiction and hence we must have $S$ empty. This completes the proof of the first statement. Assume now that $a\in \dR$. The proof of the first statement implies that $x(t)<0$ for all $t\in (a,b]$. By continuity we deduce that $x(a)\le 0$. We need to prove that $x(a)<0$.
Assume that, on the contrary, $x(a)=0$. Using the Mean Value Theorem gives, for some $s\in (a,b)$:
\[
0 > x(b)=x(b) - x(a)= \dot{x}(s)(b-a)>0,
\]
where we used  \eqref{sign} for $t=s$ and the fact that $x(s)<0$, so $\dot{x}(s)>0$. The above expression entails a contradiction and hence we must have $x(a)<0$.
\end{proof}

\begin{lemma}[Sign of a Trajectory]  \label{lem:3}
Consider a monic depressed quartic polynomial, with $a_2<0$ and $a_1\neq0$.  Let $t_0 = \sqrt{-a_2/2}$\ \ and\ \ $x_0 = -\sqrt[3]{a_1/4}$\,.  Consider the initial value problem \eqref{ODE_valley_quartic}. Let $x(\cdot)$ be the maximally extended solution of \eqref{ODE_valley_quartic}, and $(m_0,t_0]$ the corresponding maximal interval of definition of $x(\cdot)$. Then $m_0=-\infty$ and $\sgn(x(t)) = -\sgn(a_1)$ for all $t\in(-\infty,t_0]$.
\end{lemma}
\begin{proof}
Suppose that $a_1>0$, and hence $x_0<0$. The definition of $m_0$ indicates that the right-hand side of the ODE in \eqref{ODE_valley_quartic} is not zero and doesn't change sign over $[m_0,t_0]$. The choice of $x_0$ and $t_0$ imply that the denominator in the right-hand side of the ODE is positive at $t=t_0$. Hence we must have that this denominator is positive over $[m_0,t_0]$. This fact implies that property \eqref{sign} holds for the ODE  \eqref{ODE_valley_quartic} in the interval $(a,b]:=(m_0,t_0]$.  Since $x(t_0)=x_0<0$ we can apply Lemma \ref{sign-ODE} to conclude that  $x(t)<0$ for all $t\in\,[m_0,t_0]$, where $m_0\in \dR\cup\{-\infty\}$. 

Next, we prove that the solution $x(\cdot)$ can be infinitely extended to the left, in other words, $m_0 = -\infty$.  Suppose that, on the contrary, $m_0\in \mathbb{R}$. By Lemma~\ref{lem:solvability}(c), this can only happen if the right hand side of the ODE in~\eqref{ODE_valley_quartic} becomes discontinuous at $t=m_0$. This implies that 
\begin{equation}
  \label{eq:5}
  \mu_{xx}(x(m_0),m_0) = 0\,.
\end{equation}
By Lemma \ref{lem:1}, we have 
\[
\mu_x(x(t),t) = 0\,,
\]
for all $t\in [m_0,t_0]$. Therefore,
\begin{equation}
  \label{eq:6}
\mu_x(x(m_0),m_0) = 0\,.
\end{equation}
By Lemma~\ref{flatness}, Equations \eqref{eq:5}--\eqref{eq:6} have a unique solution with $x(m_0) = \sqrt[3]{a_1} / 2 > 0$. This is in contradiction with the second statement in Lemma \ref{sign-ODE}, which asserts that $x(m_0)<0$. Hence we must have $m_0=-\infty$.  The proof for the case when $a_1<0$ is obtained similarly. Namely, in this case we use that $x_0>0$ and Lemma \ref{sign-ODE} must be used for this case.
\end{proof}
%

\begin{theorem}[Well-definedness Yielding Global Minimizer]  \label{well-defined}
For a monic depressed quartic polynomial, with $a_1\neq0$ and $a_2<0$, Algorithm~\ref{algo2} is well-defined and it yields the global minimizer. 
\end{theorem}
\begin{proof}
By Step 1 of Algorithm~\ref{algo2}, $t_0 = \sqrt{-a_2/2}$\ \ and\ \ $x_0 = -\sqrt[3]{a_1/4}$\,. Since $a_2 < 0$ and $a_1\neq0$, by Lemma~\ref{lem:3}, Step~2 results in $\sgn(x(t)) = -\sgn(a_1)$ for all $t\in(-\infty,t_0]$.  Moreover, by Lemma~\ref{lem:1}, we have that
\[ 
\mu_x(x(t),t) = 0\,,\ \ \mu_{xx}(x(t),t) > 0\,,\ \forall t\in (-\infty,t_0]\,. 
\]
Therefore, $x(0)$ is a local minimizer. Now, Lemma~\ref{global_sign} and the fact that $\sgn(x(0)) = -\sgn(a_1)$ imply that  $x(0)$ must be the global minimizer. 
\end{proof}

\begin{remark} \rm 
By Lemma \ref{f: quasi-convex}, if $-3\,a_1^{2/3}/2 < a_2 < 0$, then the monic depressed quartic polynomial is quasi-convex. If  $a_2>0$ then $f$ is convex. In either case,  Algorithm~\ref{algo2} is not necessary.
\proofbox
\end{remark}

\begin{corollary}
Algorithm~\ref{algo2} is well-defined and convergent for any monic quartic polynomial.
\end{corollary}
\begin{proof}
Any quartic polynomial can be obtained from a depressed quartic polynomial through horizontal translation, or shift (and vice versa).  Therefore, Lemma~\ref{shift_lemma} on scale-shift invariance of Algorithm~\ref{algo1}, furnishes the proof.
\end{proof}

\subsection{Three types of trajectories}

Clearly, a monic quartic polynomial $f(x)$ has at most three local extrema, given by the roots of $f'(x)$.  If the roots of $f'(x)$ are distinct, then they correspond to two local minimizers and one local maximizer of $f(x)$.  Figure~\ref{curve_types} considers three cases in which (a) $f'(x)$ has a single real root, (b) $f'(x)$ has symmetric real roots and (c) $f'(x)$ has nonsymmetric real roots.  The trajectories run "forward" from each of these roots are depicted in Figure~\ref{curve_types}, providing a full characterization, on the surface defined by \eqref{mu}. 

The case when the root of $f'(x)$ is unique is exemplified in Figure~\ref{curve_types}(a):  with $f(x) = x^4 - 0.09\,x^2 - 0.03\,x - 1$,  the minimizer of $f'(x)$ is $x\approx 0.304668$.  In this case, $a_1 = -0.03\neq0$ and $a_2 = -0.09 < 0$.  It is easily checked that the conclusion of Lemma~\ref{lem:3} is satisfied, in that $\sgn(x(t)) = -\sgn(a_1)$ for all $t\in[0,t_0]$.

The case when $f'(x)$ has more than one (distinct) real root is exemplified in Figures~\ref{curve_types}(b)-(c).  In Figure~\ref{curve_types}(b), we have $f(x) = x^4 - 0.98\,x^2 + 1$, where $a_1=0$, so $f(x)$ has two global minimizers, $x=-0.7$ and $x=0.7$, and the local maximizer $x=0$.  This case (when $a_1=0$) is trivial, for which there is no need to implement Algorithm~\ref{algo2}.

On the other hand, the case when $f'(x)$ has more than one (distinct) real root, and these roots are nonsymmetric, is exemplified in Figure~\ref{curve_types}(c), with the polynomial $f(x) = x^4 - 4\,x^3/15 - 0.82\,x^2 + 0.168\,x + 1$.  The polynomial $f(x)$ has its global minimum at $x=-0.6$ and a local minimum at $x=0.7$.  The local maximizer of $f(x)$ is $x=0.1$.  One can easily verify Lemma~\ref{lem:3}, in that $\sgn(x(t)) = -\sgn(a_1)$ for all $t\in[0,t_0]$.  In this case, the system $\mu_x(x,t) = 0$\ \ and\ \ $\mu_{xx}(x,t)=0$ has a solution by Lemma~\ref{flatness}, which is shown at the bottom left in Figure~\ref{curve_types} as the point where two of the trajectories emanating from the local maximum and local minimum points merge on the surface.

 \begin{figure}
 \begin{minipage}{80mm}
 \begin{center}
 \psfrag{x}{$x$}
 \psfrag{t}{$t$}
 \includegraphics[width=80mm]{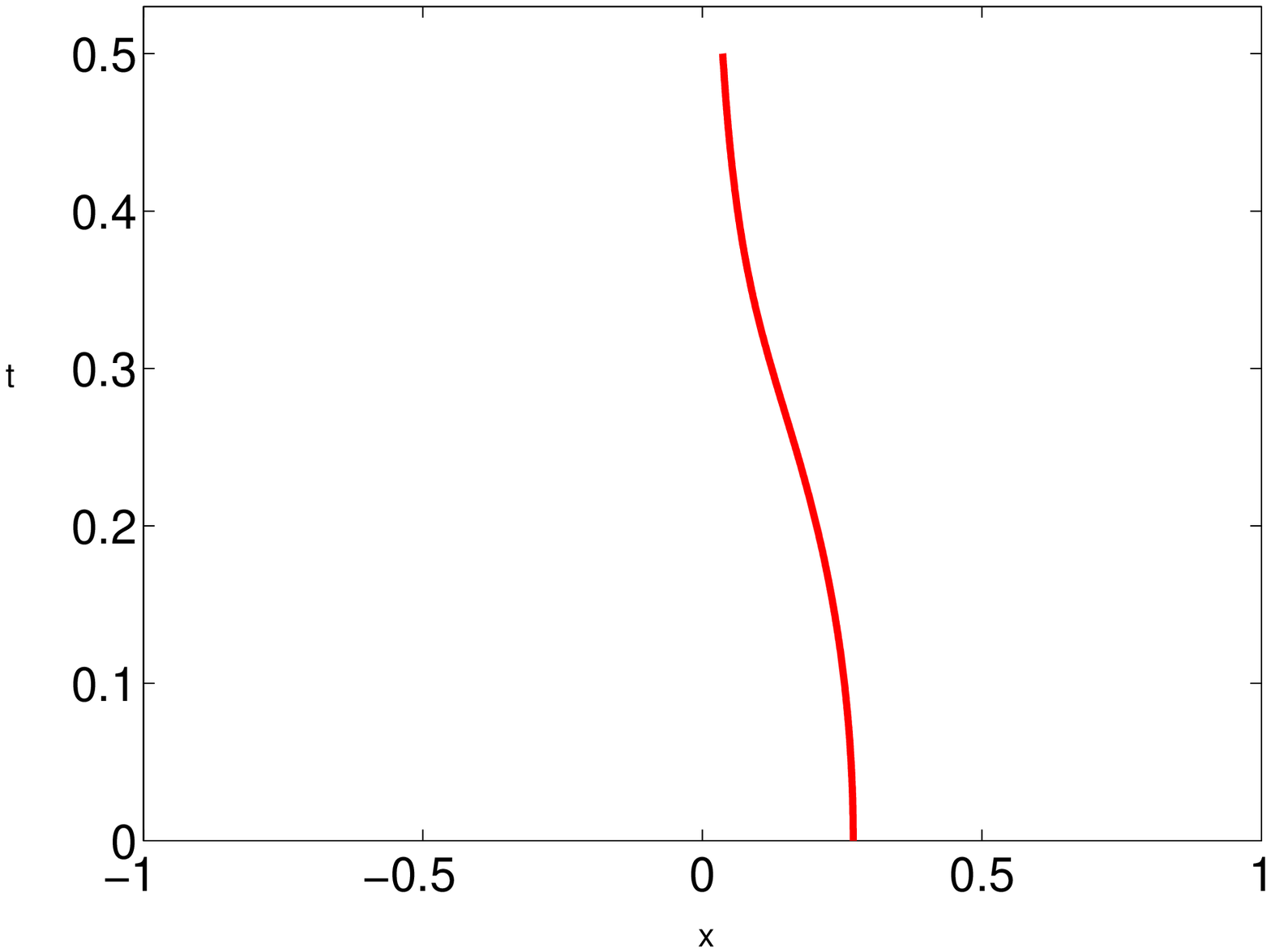} 
 \end{center}
 \end{minipage}
 \hspace{-5mm}
 \begin{minipage}{80mm}
 \begin{center}
 \psfrag{x}{$x$}
 \psfrag{t}{$t$}
 \psfrag{mu}{$\mu(x,t)$}
 \includegraphics[width=80mm]{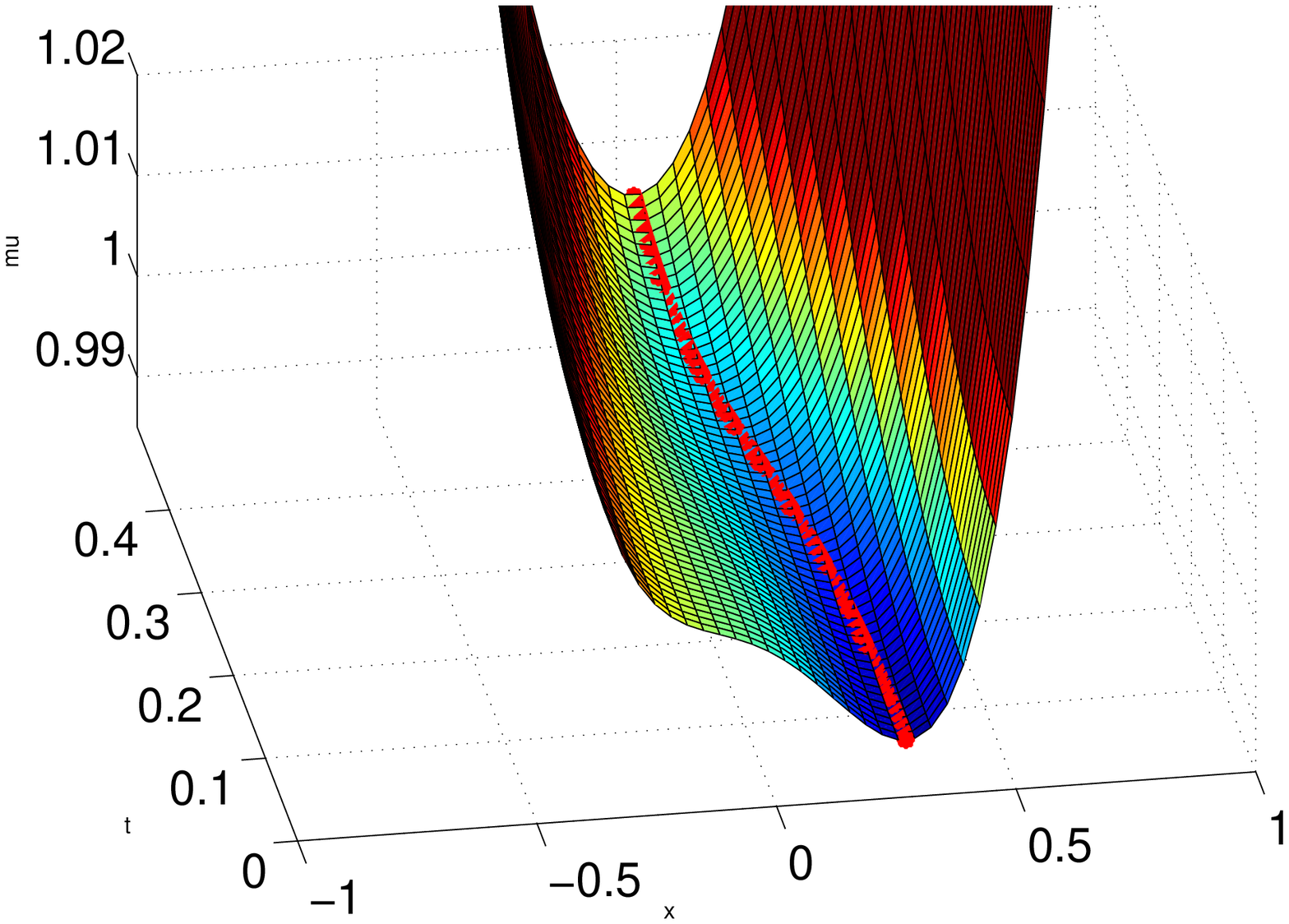}
 \end{center}
 \end{minipage} \\[3mm]
 \begin{center}
 (a) A quasi-convex $f$: $f(x) = x^4 -
 0.09\,x^2 - 0.03\,x - 1$. \\[3mm] 
 \end{center}
 \begin{minipage}{80mm}
 \begin{center}
 \psfrag{x}{$x$}
 \psfrag{t}{$t$}
 \includegraphics[width=80mm]{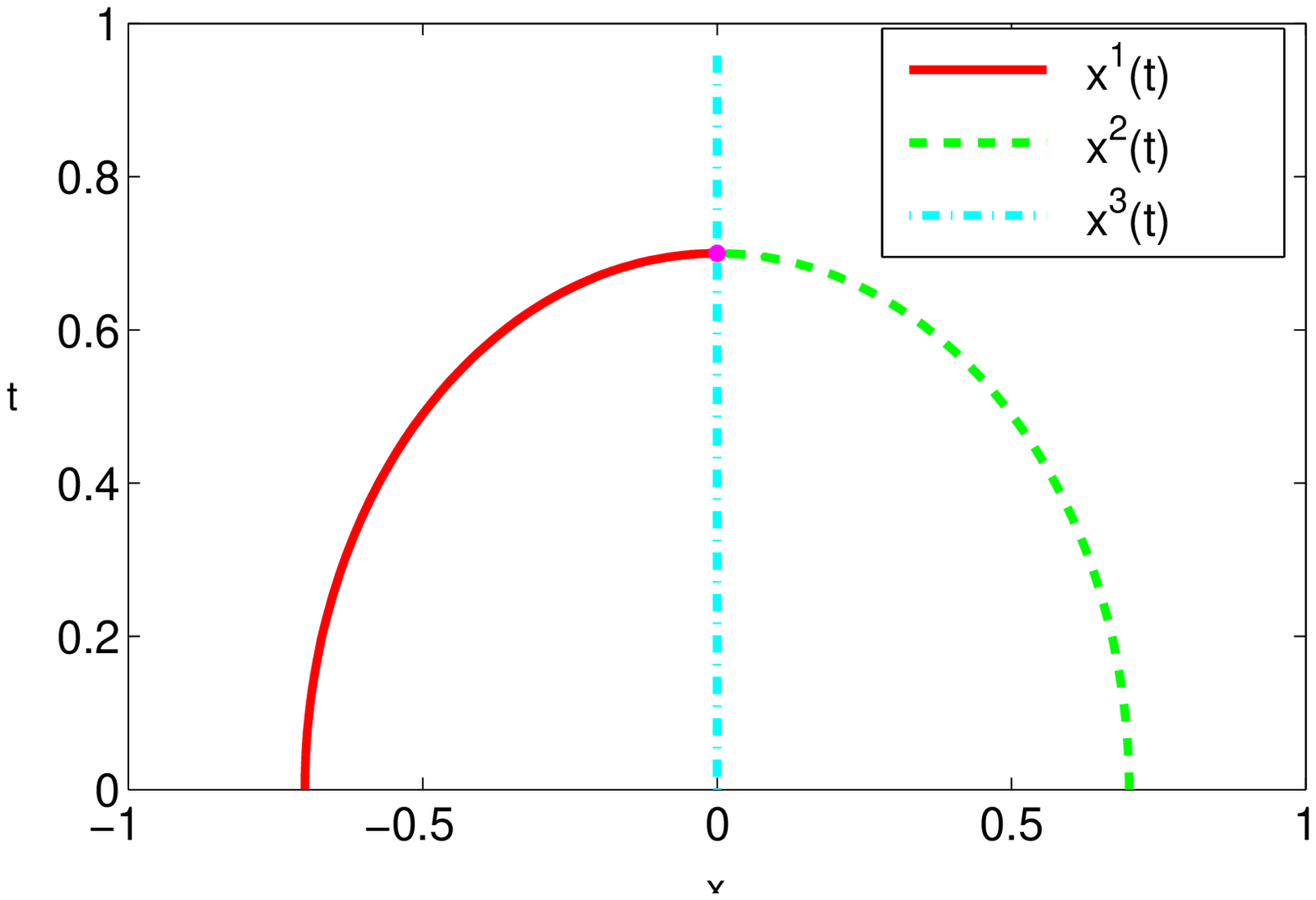}
 \end{center}
 \end{minipage}
 \hspace{-5mm}
 \begin{minipage}{80mm}
 \begin{center}
 \psfrag{x}{$x$}
 \psfrag{t}{$t$}
 \psfrag{mu}{$\mu(x,t)$}
 \includegraphics[width=80mm]{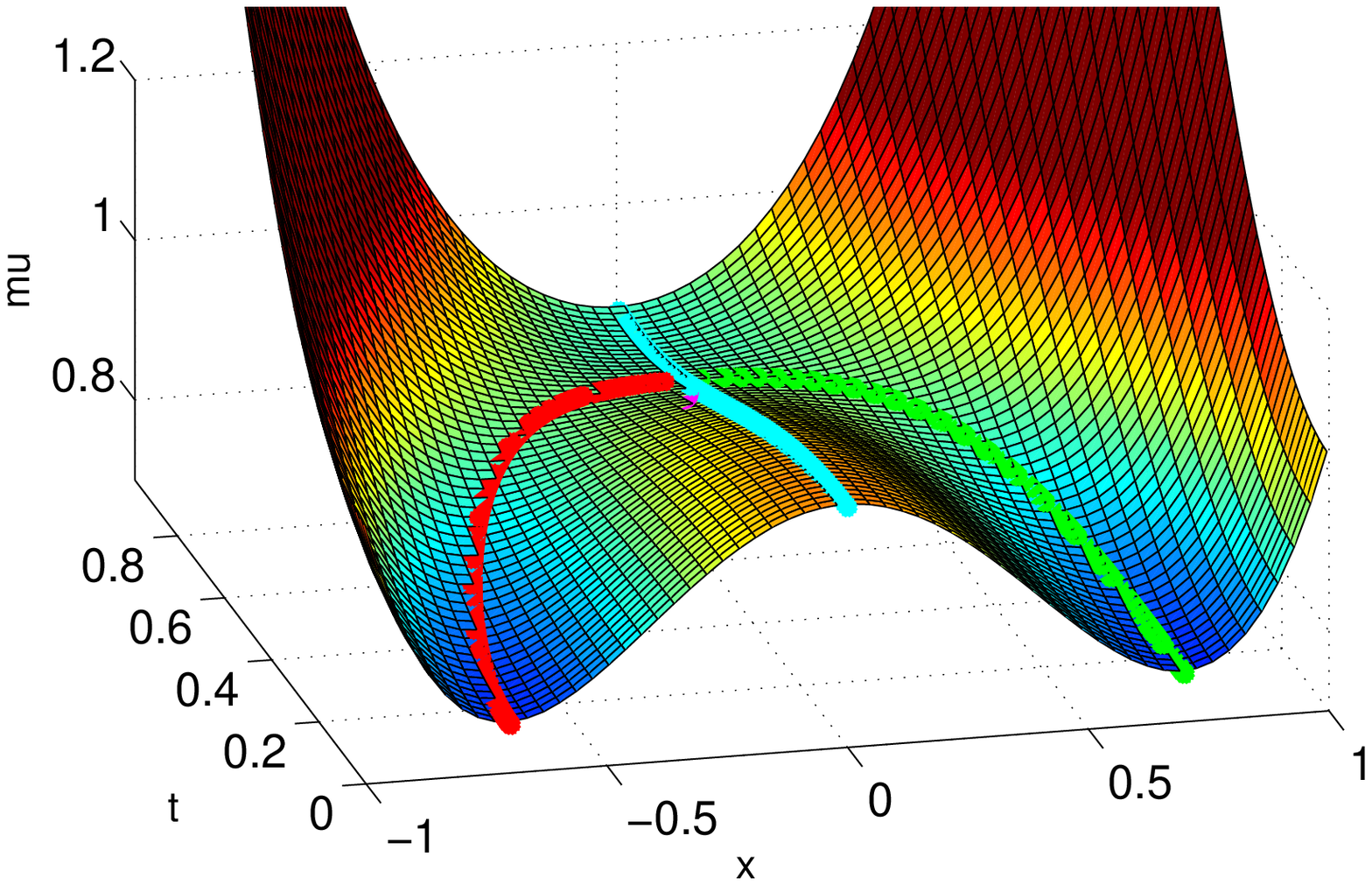}
 \end{center}
 \end{minipage} \\[3mm]
 \begin{center}
 (b) A nonconvex $f$ with symmetric roots: $f(x) = x^4 -
 0.98\,x^2 + 1$. 
 \\[3mm] 
 \end{center}
 \begin{minipage}{80mm}
 \begin{center}
 \psfrag{x}{$x$}
 \psfrag{t}{$t$}
 \includegraphics[width=80mm]{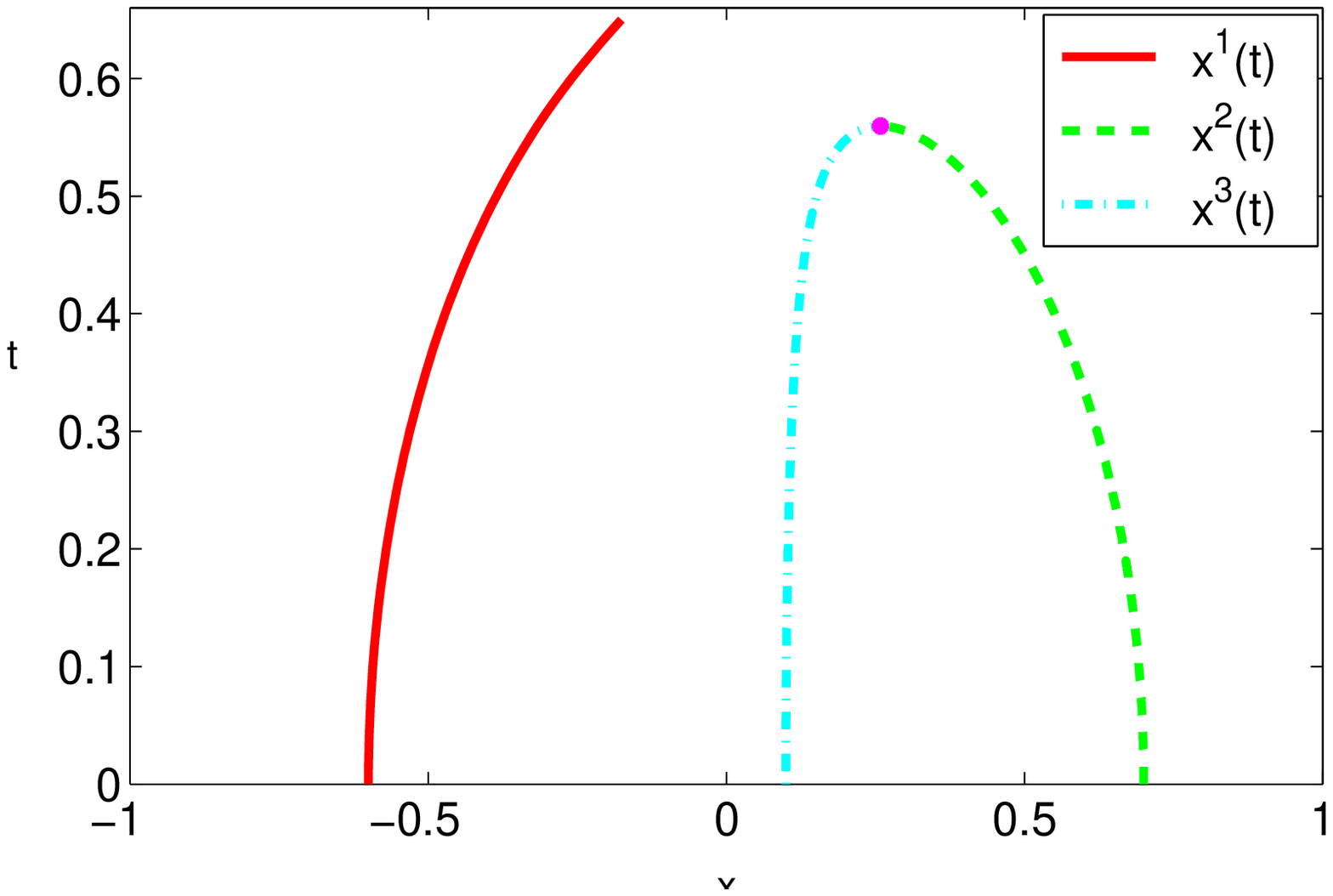} 
 \end{center}
 \end{minipage}
 \hspace{-5mm}
 \begin{minipage}{80mm}
 \begin{center}
 \psfrag{x}{$x$}
 \psfrag{t}{$t$}
 \psfrag{mu}{$\mu(x,t)$}
 \includegraphics[width=80mm]{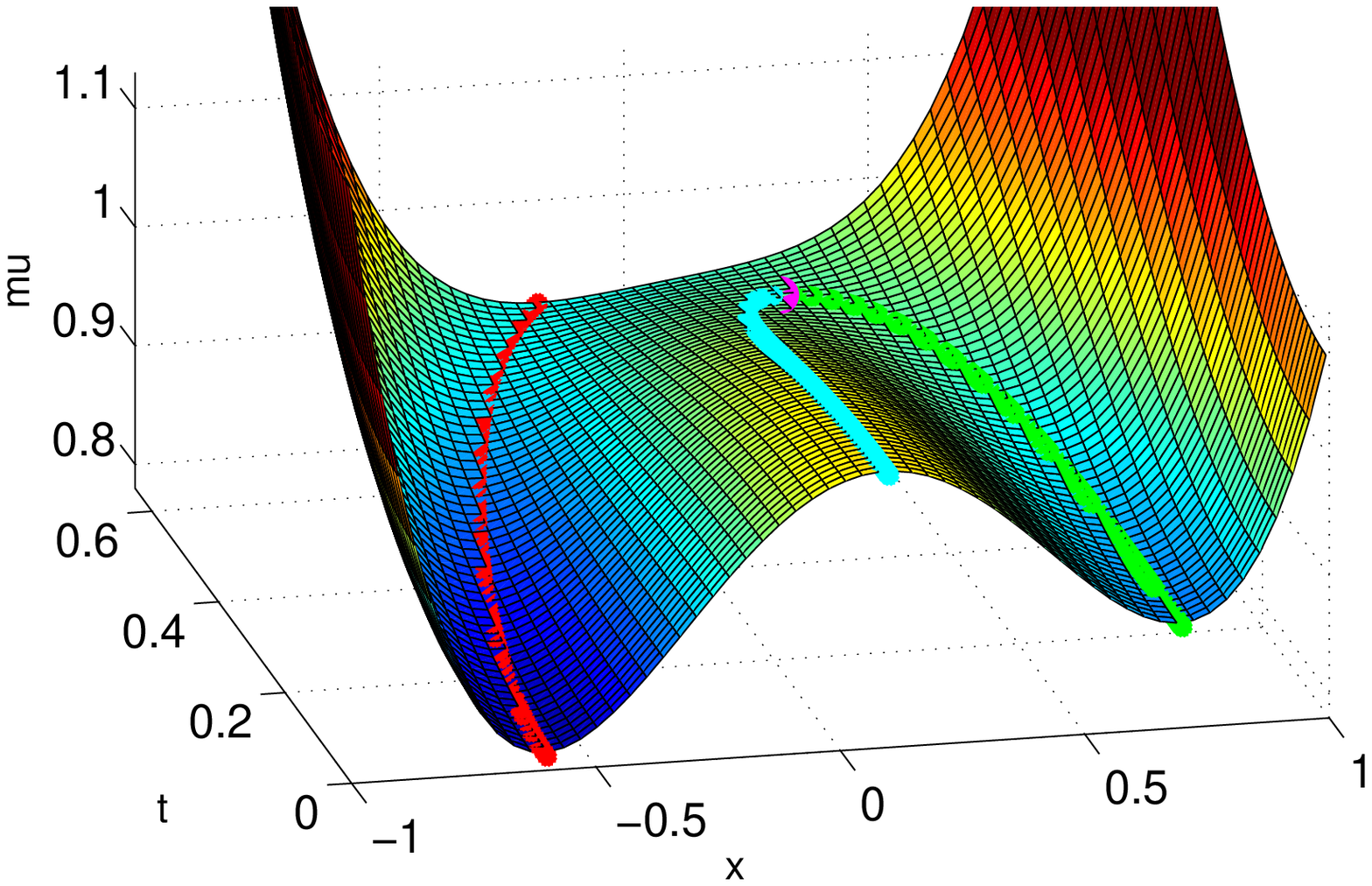} 
 \end{center}
 \end{minipage} \\[3mm]
 \begin{center}
 (c) A nonconvex $f$ with nonsymmetric roots:
 $f(x) = x^4 - 4\,x^3/15 - 0.82\,x^2 + 0.168\,x + 1$.
 \\[3mm] 
 \end{center}
 \caption{\sf  Examples for three types of polynomials and their associated trajectories.}
 \label{curve_types}
 \end{figure}

\section{Trajectory Methods with Quadratic Regularization}
\label{quad_reg}

The regularization idea for polynomial optimization is not new.  Such an approach, although not explicitly stated as a
regularization, is employed in~\cite{ZhuZhaLiu2014}, for finding a global minimizer of a monic polynomial $f$ of even
degree. In~\cite{ZhuZhaLiu2014}, in an algorithm similar to Algorithm~\ref{algo1}, the function
\begin{equation}  \label{phi}
\varphi(x,t) = f(x) + \frac{t}{2}\,x^2\,,
\end{equation}
is effectively used, instead of $\mu(x,t)$.  We refer to $\varphi$ as the {\em quadratic regularization} of $f$. A direct computation from \eqref{phi} yields the partial derivatives
\begin{equation}  \label{partial-derivs}
\varphi_x(x,t) = f'(x) + t\,x\,,\quad \varphi_{xx}(x,t) = f''(x) + t\,,\quad\varphi_{tx}(x,t) = x\,.
\end{equation}

The formula for $ \varphi_{xx}$ directly gives a well-known convexity result analogous to Theorem~\ref{convexity}, for the regularization $\varphi(\cdot,t)$. 

\begin{remark}\label{conv-QR}
Assume that $f$ is twice differentiable, and that $l_0:=\inf_{x\in \dR} f''(x)\in \dR$. Then,
$\varphi(\cdot,t)$ is convex for $t\ge -l_0$.
\end{remark}

This remark can be used in Steps~1--2 to give Algorithm~\ref{algo3},  companion of Algorithm~\ref{algo1}, for the quadratic regularization.

\newpage

\begin{algorithm} \label{algo3} \
\begin{description}
\vspace*{-3mm}
\item[Step \boldmath{$1$}] Choose the parameter $t_0>0$ large enough so that $\varphi(\cdot,t_0)$ is convex. Find the (global) minimizer $x_0$ of $\varphi(\cdot,t_0)$, i.e., solve $\varphi_x(x_0,t_0) = 0$ for $x_0$.
\item[Step \boldmath{$2$}] Solve the initial value problem
\begin{equation}  \label{ODE_valley4}
\dot{x}(t) = -\frac{\varphi_{tx}(x(t),t)}{\varphi_{xx}(x(t),t)}= - \dfrac{x(t)}{f''(x)+t}\,,\quad\mbox{ for a.e. } t\in(0,t_0]\,,\quad\mbox{with } x(t_0) = x_0\,.
\end{equation}
\item[Step \boldmath{$3$}] Report $x(0)$ as a global minimizer of $f(\cdot)$.
\end{description}
\end{algorithm}

The quadratic regularization defined in \eqref{phi} is not scale-shift invariant, in the sense of Lemma \ref{shift_lemma}. This fact has been established in \cite{AriBurKay2015} by means of an example. Namely, if Algorithm~\ref{algo3} is applied to a general quartic polynomial (not depressed) then it may not yield a global minimizer.  We further illustrate this fact by means of example polynomials, including those of higher degrees, in the next section.

\section{Numerical Experiments}

In this section, via numerical experiments, we illustrate the working of our trajectory method devised utilizing Steklov regularization, i.e., Algorithm~\ref{algo1} (which becomes Algorithm~\ref{algo2} for the case in which $f$ is a quartic polynomial), on example problems involving quartic and higher-degree polynomials, as well as an example involving a non-polynomial function.  We provide comparisons with the trajectory method in~\cite{ZhuZhaLiu2014}, namely Algorithm~\ref{algo3}, which, as pointed in Section~\ref{quad_reg}, can be derived using a quadratic regularization.

We illustrate the behaviour of the algorithms by means of graphs. In Figures~\ref{fig:deg4}--\ref{fig:deg10-20} for the polynomial examples presented in this paper, the graphs in parts~(a) and (c) of the figures provide the ``contours of $t$," i.e., the graph of the regularization function (quadratic or Steklov) with a number of fixed values of $t$ between 0 and a chosen value of $t_0$.  In parts~(b) and (d) of the figures, a surface plot of the regularizing function (quadratic or Steklov) is provided.  In the figure for the non-polynomial example considered in Subsection~\ref{nonpoly}, similar graphs are displayed.

In Subsection~\ref{stats}, we measure the performance of Algorithms~\ref{algo1} and \ref{algo3} for randomly generated polynomials of certain degrees.  Table \ref{failure_rates} shows that  Algorithm~\ref{algo1} is always convergent for the quartic polynomials generated randomly, in line with Theorem~\ref{well-defined}, and convergent for the great majority of the higher-degree polynomials generated randomly.  It is further observed that, although Algorithm~\ref{algo1} does not converge for all the tested polynomials of degree greater than four, it clearly outperforms Algorithm~\ref{algo3}. 

In all graphs, the trajectory, or the solution curve of an ODE, constructed by an algorithm is also depicted.  A trajectory is generated by solving the pertaining initial value problem (IVP) using the {\sc Matlab} function {\tt ode15s}, with {\tt RelTol = 1e-08}.  We have used {\tt ode15s}, which is a choice for stiff ODEs, since only then it was possible to get a solution of the IVP or a message saying that it was not possible to get a solution, the latter being useful in obtaining the success rates of the algorithms in Subsection~\ref{stats}.

\newpage

\subsection{A quartic polynomial}
\label{ex1}

Consider minimization of the polynomial
\[
f(x) = x^4 - 8\,x^3 - 18\,x^2 + 56\,x\,,
\]
which has local minima at $x=-2$ and $x=7$ and a local maximum at $x=1$.  Note that $f(-2)=-104$, $f(7)=-833$ and $f(1)=31$.  Therefore, $x=7$ is the global minimizer of $f(x)$.  This polynomial is provided in~\cite{AriBurKay2015} as a counterexample to prove that the trajectory approach in~\cite{ZhuZhaLiu2014} using quadratic regularization, i.e., Algorithm~\ref{algo3} given in the present paper, does not necessarily yield to a global minimizer, as opposed to the claim in~\cite{ZhuZhaLiu2014}.  Indeed, as Figure~\ref{fig:deg4}(a)--(b) illustrates, the trajectory constructed by the quadratic regularization converges to the local minimizer $x=-2$ rather than the global minimizer $x=7$. As discussed in~\cite{AriBurKay2015}, the quadratic regularization function $\varphi(\cdot,t_0)$ convexifies the given quartic polynomial with $t_0 > 84$.  For visual convenience in Figure~\ref{fig:deg4}(a)--(b), we have used $t_0 = 100$ in Algorithm~\ref{algo3}, as in~\cite{AriBurKay2015}.  Again, from~\cite{AriBurKay2015}, the corresponding $x_0\approx -0.6812$.




\begin{figure}[t]
\begin{minipage}{80mm}
\begin{center}
\includegraphics[width=80mm]{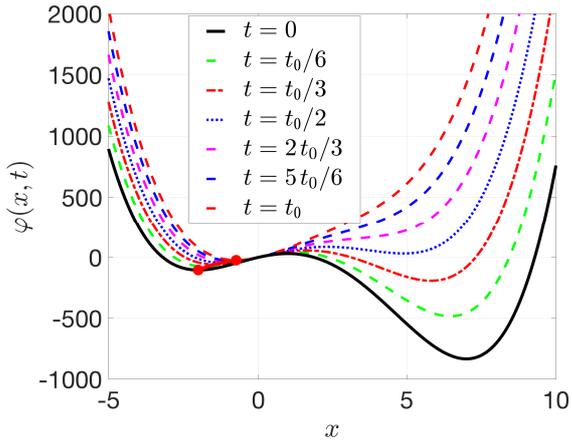} \\[3mm]
(a) Quadratic regularization -- contours of $t$ \\ with $t_0 = 100$.
\end{center}
\end{minipage}
\begin{minipage}{80mm}
\begin{center}
\includegraphics[width=80mm]{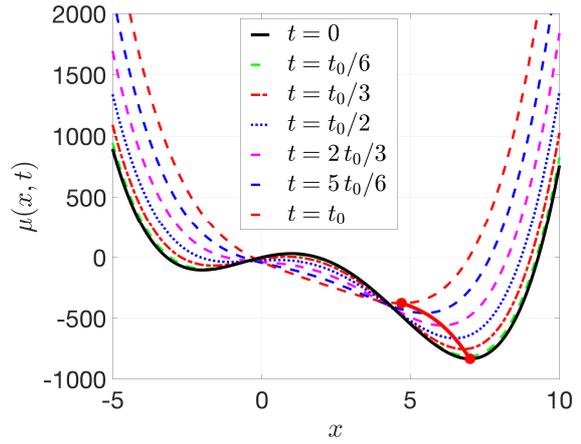} \\[3mm]
(c) Steklov regularization -- contours of $t$ \\ with $t_0 =  5$.
\end{center}
\end{minipage}
\\[5mm]
\begin{minipage}{80mm}
\begin{center}
\includegraphics[width=80mm]{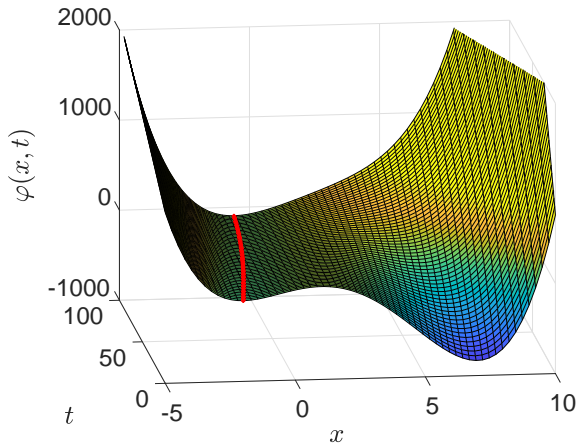} \\[3mm]
(b) Quadratic regularization -- surface \\ with $t_0 = 100$.
\end{center}
\end{minipage}
\begin{minipage}{80mm}
\begin{center}
\includegraphics[width=80mm]{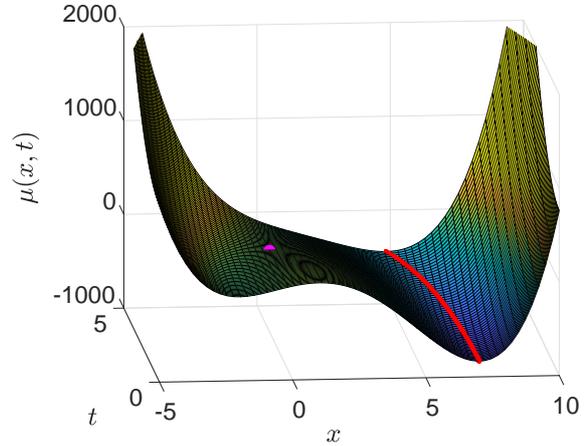} \\[3mm]
(d) Steklov regularization -- surface \\ with $t_0 =  5$.
\end{center}
\end{minipage}
\caption{\sf Trajectory methods for the quartic polynomial, $f(x) = x^4 - 8\,x^3 - 18\,x^2 + 56\,x$\,.} 
\label{fig:deg4}
\end{figure}

\newpage

Algorithm~\ref{algo1} yields the global minimizer, as expected by Theorem~\ref{well-defined}, see Figure~\ref{fig:deg4}(c)--(d).  The polynomial $f(x)$ can be rewritten in a depressed form using the transformation $x = z+2$ as
\[
f(z) = z^4 - 42\,z^2 - 80\,z - 8\,.
\]
By using Lemma~\ref{t0_x0}, we see that the Steklov regularization convexifies the given quartic polynomial just for $t_0=\sqrt{21}\approx 4.5826$, for which $x_0 = 2+\sqrt[3]{20}\approx 4.7144$.  Again for visual convenience we have used $t_0 = 5$.

In fact, by the Flatness Lemma~\ref{flatness}, the Steklov function $\mu(\cdot,t)$ is quasi-convex at 
$\widehat{t} \approx 2.6599$.  We note that $\mu_x(\widehat{x},\widehat{t}) = 0 = \mu_x(\widehat{x},\widehat{t})$, with $(\widehat{x},\widehat{t}) \approx (2.6599, -0.1544)$, which is also indicated with a (pink) mark in Figure~\ref{fig:deg4}(d).

One could as well have used $t_0 = \widehat{t} \approx 2.66$ for which $\mu(\cdot,t_0)$ is quasi-convex, and Algorithm~\ref{algo1} can be run with $(x_0,t_0)$ and the associated initial condition $\mu_x(x_0,t_0)=0$, in Step~1.

As will be seen also with the higher-order polynomials, the Steklov function $\mu(\cdot,t_0)$ is convex with a rather small $t_0$.  On the other hand, the quadratic regularization function $\varphi(\cdot,t_0)$ becomes convex with a much larger $t_0$, which is almost 20 times the $t_0$ needed for the Steklov function.  In the subsequent subsections, it will be observed that $t_0$ grows greatly with the degree of a polynomial.  When solving an IVP, a big $t_0$ makes the time span (or time horizon) $[0,t_0]$ big and this causes ODE solvers to take a much longer time and run more often into difficulties.


\afterpage{\clearpage}

\subsection{A degree-6 polynomial}
\label{ex2}

Consider minimization of the degree-6 polynomial
\[
f(x) = x^6 - 66\,x^5/5 - 9\,x^4/2 + 422\,x^3 - 474\,x^2 - 2160\,x\,,
\]
which has local minima at $x=-4,2$ and 9 and local maxima at $x=-1$ and 5.  A graph of the polynomial can be seen in Figure~\ref{fig:deg6}.  The global minimizer of $f(x)$ is $x=9$, with $f(9) = -27726.3$ (exactly).  One has the local minima $f(-4) = -9491.2$ and $f(2) = -3270.4$, and the local maxima $f(-1) = 1273.7$ and $f(5) = 1662.5$, all exact.

\begin{figure}[t]
\begin{minipage}{80mm}
\begin{center}
\includegraphics[width=80mm]{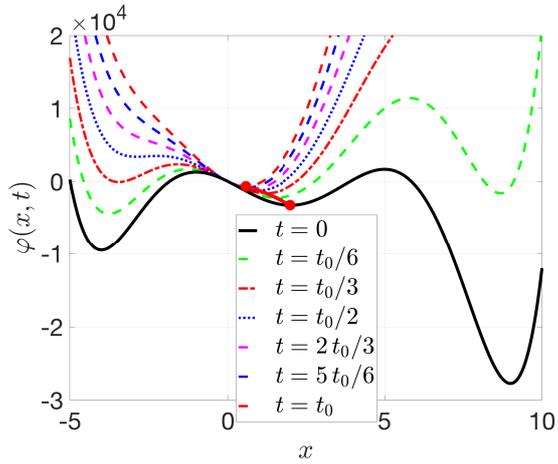} \\[3mm]
(a) Quadratic regularization -- contours of $t$ \\ with $t_0 = 4000$.
\end{center}
\end{minipage}
\begin{minipage}{80mm}
\begin{center}
\includegraphics[width=80mm]{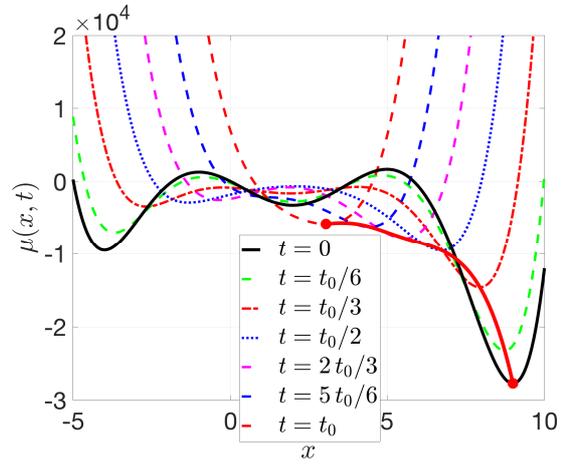} \\[3mm]
(c) Steklov regularization -- contours of $t$ \\ with $t_0 =  7$.
\end{center}
\end{minipage}
\\[5mm]
\begin{minipage}{80mm}
\begin{center}
\includegraphics[width=80mm]{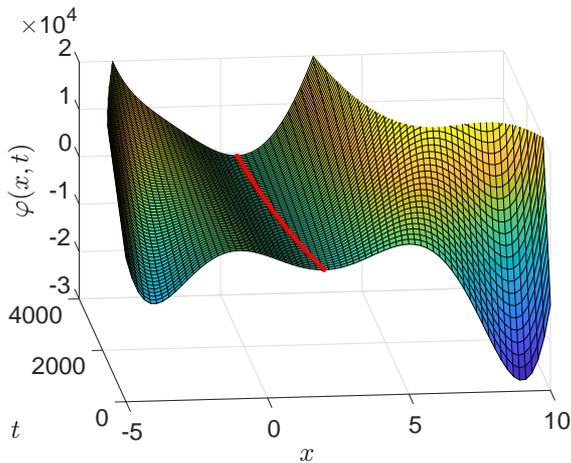} \\[3mm]
(b) Quadratic regularization -- surface \\ with $t_0 = 4000$.
\end{center}
\end{minipage}
\begin{minipage}{80mm}
\begin{center}
\includegraphics[width=80mm]{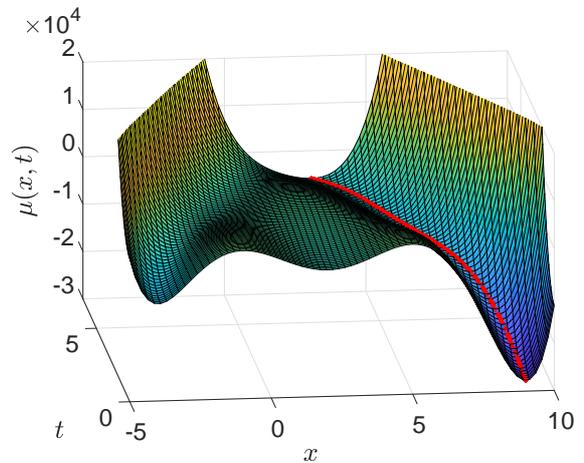} \\[3mm]
(d) Steklov regularization -- surface \\ with $t_0 =  7$.
\end{center}
\end{minipage}
\caption{\sf Trajectory methods for the degree-6 polynomial, $f(x) = x^6 - 66\,x^5/5 - 9\,x^4/2 + 422\,x^3 - 474\,x^2 - 2160\,x$\,.} 
\label{fig:deg6}
\end{figure}

We observe that $t_0=7$ is enough to convexify the Steklov function $\mu(\cdot,t_0)$, while the quadratic regularization function $\varphi(\cdot,t_0)$ requires $t_0 \approx 4000$ to become convex.  Moreover, Algorithm~\ref{algo3} (using the quadratic regularization) yields the local minimizer $x=2$, while Algorithm~\ref{algo1} (using the Steklov regularization) yields the global minimizer $x=9$.

This polynomial $f(x)$ is just one degree-6 polynomial example to illustrate the working and success of Algorithm~\ref{algo1}, as well as the working and failure of Algorithm~\ref{algo3}.  Algorithm~\ref{algo1} can also fail for some degree-6 polynomials, but not as often as Algorithm~\ref{algo3} does.  As mentioned before, detailed comparisons of success rates for each of the algorithms are shown in Subsection~\ref{stats}.


\newpage

\subsection{Degree-10 and degree-20 polynomials}
\label{ex3}

Consider minimization of the degree-10 monic polynomial $f_{10}(x)$ with the coefficients
{\small
\[
[a_9, \cdots, a_0] = [260/9, 1035/4, -120, -9415, 32172, 175765/2, -1369360/3, -148560, 1209600, 0]\,.
\]}
The global minimizer of $f_{10}(x)$ is $x=9$, with $f_{10}(9) = -2077224.75$ (exactly).

\begin{figure}[t]
\begin{minipage}{80mm}
\begin{center}
\includegraphics[width=80mm]{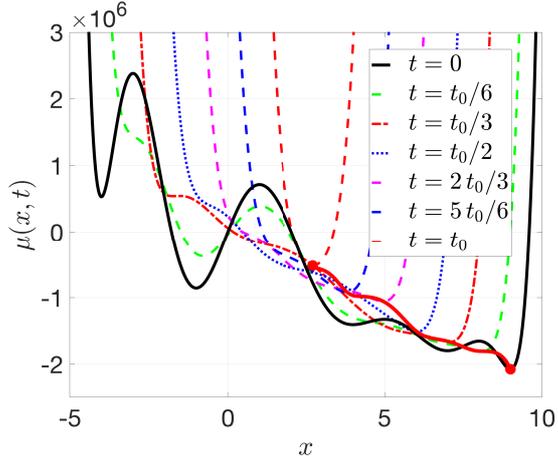} \\[3mm]
(a) Steklov regularization -- contours of $t$ \\ with $t_0 =  7$.
\end{center}
\end{minipage}
\begin{minipage}{80mm}
\begin{center}
\includegraphics[width=80mm]{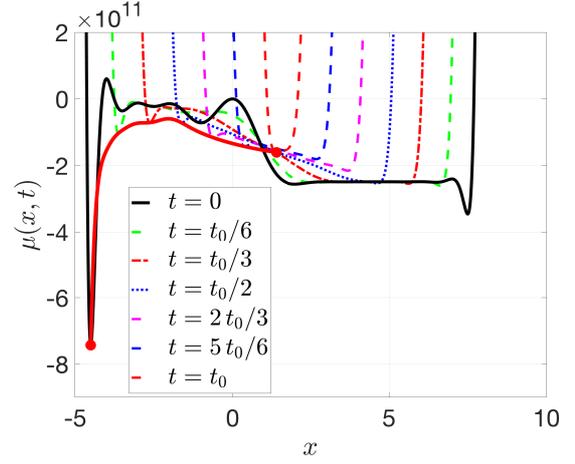} \\[3mm]
(c) Steklov regularization -- contours of $t$ \\ with $t_0 = 6$.
\end{center}
\end{minipage}
\\[5mm]
\begin{minipage}{80mm}
\begin{center}
\includegraphics[width=80mm]{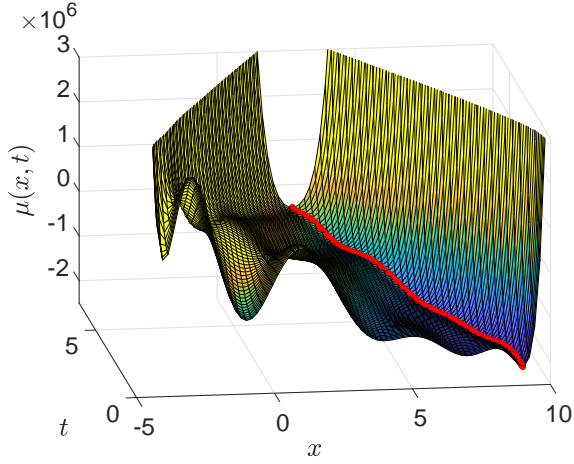} \\[3mm]
(b) Steklov regularization -- surface \\ with $t_0 =  7$.
\end{center}
\end{minipage}
\begin{minipage}{80mm}
\begin{center}
\includegraphics[width=80mm]{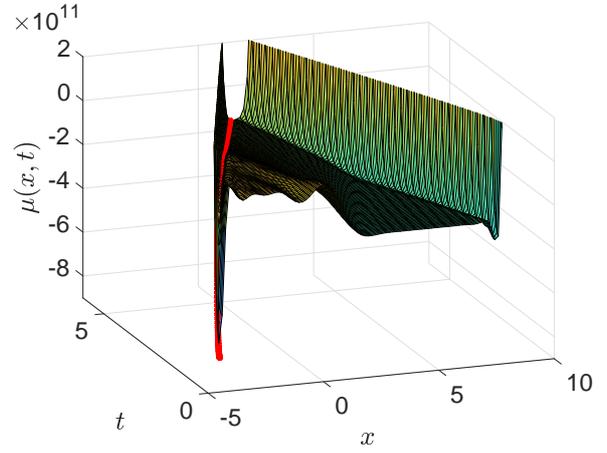} \\[3mm]
(d) Steklov regularization -- surface \\ with $t_0 =  6$.
\end{center}
\end{minipage}
\caption{\sf Algorithm~\ref{algo1} for some given degree-10 polynomial (parts~(a) and (b)) and degree-20 polynomial (parts~(c) and (d)).} 
\label{fig:deg10-20}
\end{figure}

We also consider minimization of the degree-20 monic polynomial $f_{20}(x)$ with the coefficients
{\small
\begin{eqnarray*}
[a_{19}, \cdots, a_0] &=& [680/19, 3935/9, -15755/17, -196105/8, 2230697/12, 20765145/112, \\
&& -1351162585/208, 10221013715/768, 6382409515/64, -12625444643/32, \\
&& -200463718805/288, 2498521767895/512, 465297612345/448, -2045419187205/64, \\
&& 198942566751/16, 3627285358725/32, -56515087125, -201131555625, 0, 0]\,.
\end{eqnarray*}}
The global minimizer of $f_{20}(x)$ is $x=-4.5$, with $f_{20}(-4.5) = -742786593463.8248$ (exactly).

Algorithm~\ref{algo1} successfully yields a global minimizer for both polynomials as can be seen in Figure~\ref{fig:deg10-20}.   Moderate sizes of $t_0$ (7 and 6, respectively) suffice in each case for convexification.

Algorithm~\ref{algo3} fails to serve the purpose for either polynomial.  For the degree-10 polynomial, it only yields the local minimizer $x = -1$, with $t_0 = 2\times10^6$, which, although quite large, is just enough for convexification.  For the degree-20 polynomial, one needs a far larger $t_0\approx 10^{12}$ for convexification; however, the ODE solver takes an indefinite amount of time and does not provide any answer, conceivably because of the very large orders of magnitude involved in the computations.  The graphs that were generated suggest that the trajectory method would yield $x=0$, which this time is a local maximizer!  For brevity, we do not provide the graphs for the quadratic regularization. 


\subsection{Performance comparisons between Algorithms~\ref{algo1} and \ref{algo3}}
\label{stats}

In Sections~\ref{ex1}--\ref{ex3}, we have applied Algorithms~\ref{algo1} and \ref{algo3} to four selected polynomials of degrees four, six, 10 and 20, and illustrated the workings of both algorithms.  Algorithm~\ref{algo1} was successful in finding a global minimum of each of the polynomials considered in Sections~\ref{ex1}--\ref{ex3}, while Algorithm~\ref{algo3} consistently failed.  In all fairness, neither Algorithm~\ref{algo1} is successful in dealing with every single polynomial (computationally speaking) nor Algorithm~\ref{algo3} is unsuccessful for every single polynomial.  To better understand how these two methods compare, we present failure rates of both algorithms for 1000 randomly generated polynomials of various degrees.  We have generated the polynomials in such a way that their extremal values were uniformly distributed over the interval $[-5,5]$.

\begin{table}
\begin{minipage}{35mm}
\
\end{minipage}
\begin{minipage}{40mm}
\hspace{10mm}{\em Algorithm~\ref{algo1}} \\[1mm]
\begin{tabular}{ccc}
$n$ & $t_0$ & Failure rate \\ \hline
4\ & 6 & 0\% \\
6\ & 7 & 1\% \\
8\ & 7 & 2\% \\
10\ & 7 & 4\% \\
12\ & 7 & 4\% \\
14\ & 7 & 4\% \\
20\ & 7 & 7\%
\end{tabular}
\end{minipage}
\hspace*{2mm}
\begin{minipage}{40mm}
\hspace{10mm}{\em Algorithm~\ref{algo3}} \\[1mm]
\begin{tabular}{ccc}
$n$ & $t_0$ & Failure rate \\ \hline
4\ & $10^{3}$ & 26\% \\
6\ & $10^{4}$ & 63\% \\
8\ & $10^{5}$ & 77\% \\
10\ & $10^{8}$ & 84\% \\
12\ & $10^{8}$ & 88\% \\
14\ & $10^{8}$ & 92\% \\
20\ & $10^{10}$ & 96\%
\end{tabular}
\end{minipage}
\caption{\sf A comparison of failure rates of Algorithms~\ref{algo1} and \ref{algo3} for degree-$n$ polynomials ($n$ as listed).}
\label{failure_rates}
\end{table}

Table~\ref{failure_rates} lists the failure rates for each algorithm as they are applied to polynomials of various degrees,  where the polynomials of each degree are randomly generated 1000 times.  We declare failure of the method for a given polynomial when either the algorithm did not converge, or it converged to a local minimum. For polynomials of degree higher than four, it is not trivial (if not impossible), to check convexity of $\mu(\cdot,t_0)$ or $\varphi(\cdot,t_0)$ for a given $t_0$.  The choice we made for the value of $t_0$ required in each of the experiments is drastically different for each method. We observe that Algorithm~\ref{algo1} requires $t_0\in\{6,7\}$ for all cases, while Algorithm~\ref{algo3} requires $t_0\in[10^{3},10^{10}]$, with values  increasing with the degree of the polynomials. These large values of $t_0$ promote convexification of $\varphi(\cdot,t_0)$,  but, at the same time, they are likely to cause numerical instabilities. In summary, perhaps not many but still some of the failures of Algorithm~\ref{algo3} may be attributed to (i) $t_0$ not being large enough for convexification, (ii) $t_0$ being too large, or both. It is likely that values of $t_0$ greater than the values of $t_0$ already listed in Table~\ref{failure_rates} (especially for high degree polynomials) will cause numerical instabilities. This is another reason why Algorithm~\ref{algo1} looks favourable, when compared with Algorithm~\ref{algo3}.   

Having made these remarks, especially for the high degree polynomials, one may consider doing a rescaling in order to avoid high orders of magnitudes in computations; however, we have not considered a rescaling of any of the polynomials in our computations. In the case of Algorithm~\ref{algo1}, there is certainly room for choosing $t_0$ to be bigger in the experiments.  

The failure rates for Algorithm~\ref{algo3} are very high, increasing sharply with degree, reaching 84--96\% for polynomials of degree 10--20.  Even for quartic polynomials, the failure rate of Algorithm~\ref{algo3} is rather high, at 26\%, while Algorithm~\ref{algo1} has no failures for this case, as expected by Theorem~\ref{well-defined}.  From the numerical experiments, we observe that Algorithm~\ref{algo1} can fail, even for degree-6 and degree-8 polynomials; however, the failure rate is small, at 1--2\%, in practical terms.  This rate is far smaller than that of Algorithm~\ref{algo3} for similar degree polynomials, as shown in Table ~\ref{failure_rates}.

\afterpage{\clearpage}

\subsection{A non-polynomial function}
\label{nonpoly}

\begin{figure}[t]
\begin{center}
\includegraphics[width=130mm]{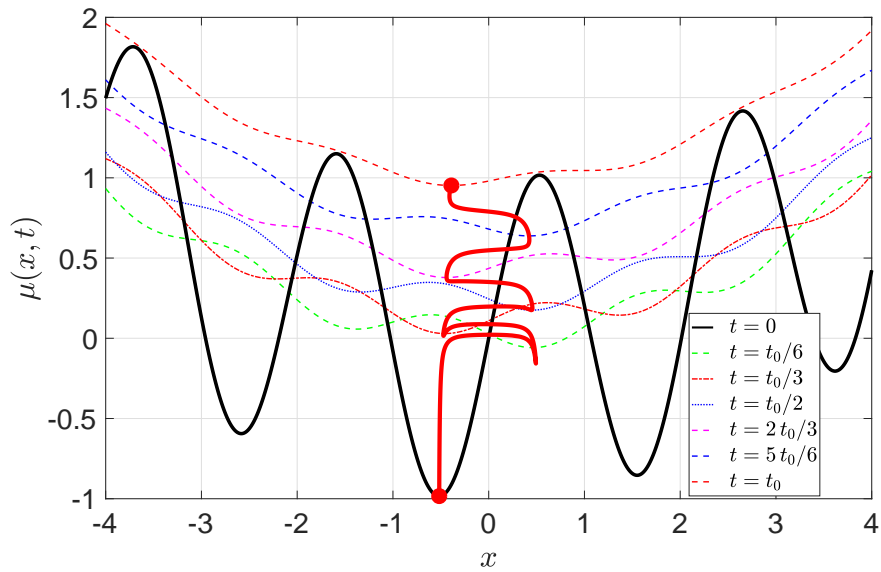} \\[3mm]
(a) Steklov regularization -- contours of $t$ \\ with $t_0 =  7$.
\end{center}
\ 
\begin{center}
\includegraphics[width=130mm]{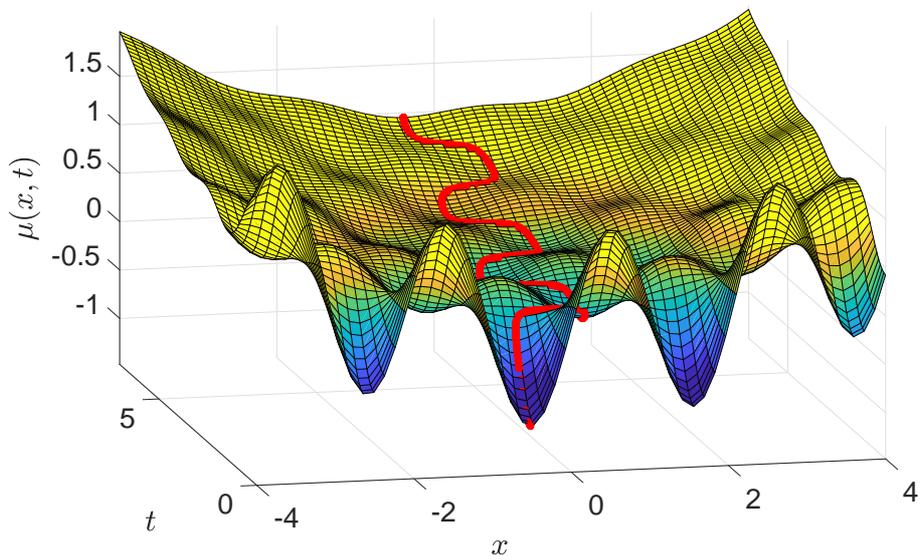} \\[3mm]
(b) Steklov regularization -- surface \\ with $t_0 =  7$.
\end{center}
\caption{\sf Algorithm~\ref{algo1} for the nonpolynomial function, $f(x)=0.06\,x^2 + \sin3x$.} 
\label{fig:nonpoly}
\end{figure}

In Sections~\ref{ex1}--\ref{stats}, we have tested the performance of Algorithms~\ref{algo1} and \ref{algo3} for polynomial functions.  In this section we consider the nonpolynomial coercive function
\[
f(x)=0.06\,x^2 + \sin3x\,,
\]
which has many local minima.  Even though this function does not satisfy the assumptions of Theorem \ref{convexity}, it is straightforward to check that the conclusion of this theorem holds for this function. To do this, we use \eqref{muxx} and some elementary algebra to derive
\[
\mu_{xx}(x,t)=0.12-3\sin3x\,\dfrac{\sin3t}{t}.
\]
Since $\lim_{t\to \infty} (\sin3t)/t=0$, we have that for every $x$ there exists a $t_0>0$ (independent of $x$), such that $\mu_{xx}(x,t)>0$ for all $t>t_0$. Thus, for those values of $t$, $\mu(\cdot,t)$ is strictly convex.  In our experiments, however, we use a $t_0$ that makes $\mu(\cdot,t_0)$ quasi-convex. Namely, with the parameter value of $t_0=7$, the Steklov function becomes quasi-convex, and its minimizer is $x_0 = -0.3896$. This minimizer is in turn used to start solving the initial value problem.  Algorithm~\ref{algo1} then finds the global minimizer as $x(0) = -0.5167$. Figures~\ref{fig:nonpoly}(a) and \ref{fig:nonpoly}(b) illustrate the $t$-contour and surface plots, respectively, as well as the trajectory constructed by Algorithm~\ref{algo1}, in each of parts (a) and (b).

\clearpage

\newpage

\section{Conclusion}

We have proposed a trajectory-based algorithm, Algorithm~\ref{algo1}, using the Steklov regularization function $\mu(x,t)$, for finding a global minimizer of univariate coercive functions. The so-called Steklov smoothing function has been previously studied in the literature, as a smoothing tool for small values of $t$. Our study considers using this function as a regularization tool. Namely,  we have proved that, for large enough $t$, $\mu$ convexifies certain univariate coercive functions.  We proved convergence of Algorithm~\ref{algo1} for quartic polynomials.  We tested it for higher-degree polynomials, as well as a non-polynomial function for illustration.  

We have made comparisons with an existing trajectory-based algorithm, reformulated here as Algorithm~\ref{algo3}, which uses a quadratic regularization instead.  Using 1000 randomly generated polynomials, we found that, for degree-6 polynomials, while the failure rate of Algorithm~\ref{algo1} is only 1\%, Algorithm~\ref{algo3} fails in 63\% of the cases.  For degree-20 polynomials, these percentages are 7 and 96, respectively, pointing to the fact that Algorithm~\ref{algo1} provides a better option.

Throughout the paper, we obtained auxiliary results (apart from convergence) regarding Algorithm~\ref{algo1}, the Steklov function, and quartic polynomials, which are worthy in their on right.

In Algorithm~\ref{algo1}, we require $t_0$ to be chosen so as to convexify $\mu$; however, one may instead require $t_0$ to {\em quasi-convexify} $\mu$, which would possibly result in an even smaller, i.e., a more desirable, $t_0$.  One should note that most of the powerful numerical methods for minimization of convex functions are also applicable to minimization of quasi-convex functions \cite{BazSheShe}.

As with any other global optimization technique, Algorithm~\ref{algo1} cannot find a global optimizer in every single situation.  However, it provides a promising and viable option for searching global minimizers of general univariate coercive functions.  A future line of investigation should concern extensions of Algorithm~\ref{algo1} to multi-variable coercive functions, which clearly has a much wider scope for theory and applications.


\begin{thebibliography}{30}

\bibitem{AriBurKay2015}
{\sc O. Ar{\i}kan, R. S. Burachik and C. Y. Kaya},
``Backward differential flow'' may not converge to a global minimizer of polynomials. 
{\em J. Optim. Theory Applic.}, {\bf 167}, 401--408, 2015. 

\bibitem{Arnold1978}
 {\sc V. I. Arnold},
{\em Ordinary Differential Equations}. 
The MIT Press, Cambridge, 1978.

\bibitem{AttChbPeypRed2018}
{\sc H, Attouch, Z. Chbani, J. Peypouquet, and P. Redont},
Fast convergence of inertial dynamics and algorithms with asymptotic vanishing viscosity.
{\em Math. Program.}, {\bf 168}(1-2), 123--175, 2018.

\bibitem{BazSheShe}
{\sc M. S. Bazaraa, H. D. Sherali and C. M. Shetti},
{\em Nonlinear Programming: Theory and Algorithms, 3rd edition}. 
Wiley InterScience, New Jersey, 2006.

\bibitem{BotCse2018}
{\sc R. I. Bo{\c t} and E. R. Csetnek},
Convergence rates for forward--backward dynamical systems associated with strongly monotone inclusions.
{\em J. Math. Anal. Applic.}, {\bf 457}(2), 1135--1152, 2018.

\bibitem{Chen2012}
{\sc X. Chen},
Smoothing methods for nonsmooth, nonconvex minimization.
{\em Math. Program., Ser. B}, {\bf 134}, 71--99, 2012.

\bibitem{ErmNorWet1995}
{\sc Y. M. Ermoliev, V. I.  Norkin, and R. J.-B. Wets},
The minimization of semicontinuous functions: mollifier subgradients. 
{\em SIAM J. Control Optim.}, {\bf 32}, 149--167, 1995.

\bibitem{GarVic2013}
{\sc R. Garmanjani, L. N. Vicente},
Smoothing and worst-case complexity for direct-search methods in nonsmooth optimization. 
{\em IMA J. Num. Anal.}, {\bf 33}, 1008--1028, 2013.

\bibitem{Gupal1977}
{\sc A. M. Gupal},
On a method for the minimization of almost-differentiable functions. 
{\em Cybernet. Syst. Anal.}, {\bf 13}, 115--117, 1977.

\bibitem{HorTuy1996}
{\sc R. Horst and H. Tuy},
{\em Global Optimization: Deterministic Approaches}.
Springer-Verlag, Berlin, Heidelberg, Germany, 1996.


\bibitem{LerSer2013}
{\sc D. Lera and Y. D. Sergeyev},
Acceleration of univariate global optimization algorithms working with Lipschitz functions and Lipschitz first derivatives.
{\em SIAM J. Optim.}, {\bf 23}(1), 508--529, 2013.


\bibitem{RocWet2004}
{\sc R. T. Rockafellar and R. J.-B. Wets},
{\em Variational Analysis}.  
Springer-Verlag, Berlin, Heidelberg, Germany, 2004.

\bibitem{Scholz2012}
{\sc D. Scholz},
{\em Deterministic Global Optimization: Geometric Branch-and-bound Methods and Their Applications}.
Springer, New York, 2012.


\bibitem{SnyKok2009}
{\sc J. A. Snyman and S. Kok},
A reassessment of the Snyman--Fatti dynamic search trajectory method for unconstrained global optimization.
{\em J. Glob. Optim.}, {\bf 43}, 67--82, 2009.

\bibitem{SW}
{\sc J. Stoer and C. Witzgall},
{\em Convexity and Optimization in Finite Dimensions I}.
Springer-Verlag, Berlin-Heidelberg, 1970.
%

\bibitem{ZhaXio2009}
{\sc X. Zhang and Y. Xiong},
Impulse noise removal using directional difference based noise detector and adaptive weighted mean filter.
{\em IEEE Signal Proc. Lett.}, {\bf 16}, 295--298, 2009.

\bibitem{ZhuZhaLiu2014}
{\sc J. Zhu, S. Zhao, and G. Liu},
Solution to global minimization of polynomials by backward differential flow.
{\em J. Optim. Theory Applic.}, {\bf 161}, 828--836, 2014.

\end{thebibliography}
\end{document}